\documentclass[11pt]{amsart}
\usepackage{txfonts}
\usepackage[T1]{fontenc}
\usepackage[utf8]{inputenc} 
\usepackage{geometry} 
\usepackage{booktabs} 
\usepackage{array} 
\usepackage{paralist} 
\usepackage{verbatim} 
\usepackage{tikz}
\usetikzlibrary{arrows}
\usetikzlibrary{decorations.markings}
\usepackage{subfig}
\usepackage{tabularx}
\usepackage{enumitem}
\usepackage{amsmath,amsthm,amscd}
\usepackage{amssymb,amsfonts}
\usepackage{fancyhdr} 
\usepackage{pgf,tikz}
\usepackage{caption}
\usepackage{tikz-cd}
\usepackage{tikzscale}
\usetikzlibrary{patterns}
\usepackage{rotating}
\usepackage{xcolor}
\usepackage{lscape}
\usepackage{xspace}
\usepackage{xfrac}
\usepackage{blindtext}
\usepackage{bm}
\usepackage[ruled,vlined]{algorithm2e}

\usetikzlibrary{decorations.markings}
\usetikzlibrary{patterns}
\usetikzlibrary{calc}
\usepackage{soul}

\usepackage{tikz}
        \usetikzlibrary{patterns,hobby}
        \usepackage{pgfplots}
        \pgfplotsset{compat=1.6}

\usepackage{faktor} 
\usepackage{pgf,tikz}
\usetikzlibrary{arrows.meta}
\usetikzlibrary{decorations.markings}
\usetikzlibrary{patterns}

  \newcommand{\calR}{{\mathcal R}}
\newcommand{\bfz}{{\bf z}}
  \newcommand{\Ab}{\operatorname{Ab}}

\usepackage{color}   
\usepackage{hyperref}
\hypersetup{
    colorlinks=true, 
    linktoc=all,     
    linkcolor=blue,  
    citecolor=blue,
    filecolor=blue,
    urlcolor=blue
}

\newcommand{\CC}{\mathbb{C}}

\newcommand\PP{{\mathbb P}}

\newcommand\calE{{\mathcal E}}

\newcommand\ZZ{{\mathbb Z}}
\newcommand\NN{{\mathbb N}}
\newcommand\calF{{\mathcal F}}

\newcommand\R{{\mathbb R}}

\newcommand{\MS}{{\rm MS}^{k}}
\DeclareMathOperator\Res{Res}

\DeclareDocumentCommand{\rec}{ O{a} O{n}}{\left[ #1 \right]^{#2}}

 \newcommand{\komoduli}[1][g]{{\Omega^{k}\mathcal M}_{#1}}
 \newcommand{\omoduli}[1][g]{{\Omega\mathcal M}_{#1}}
 \newcommand{\quadomoduli}[1][g]{{\Omega^{2}\mathcal M}_{#1}}

\usepackage[
	    style=alphabetic,
            isbn=false,
            doi=false,
            url=false,
            backend=bibtex,
            maxnames = 5,
            sorting=nty,
           ]{biblatex}
\defbibheading{bibliography}[\bibname]{%
   \section*{References}%
}

\DeclareFieldFormat[article]{title}{{\it #1}}
\DeclareFieldFormat{journaltitle}{{\rm #1}}
\renewbibmacro{in:}{%
  \ifentrytype{article}{}{\printtext{\bibstring{in}\intitlepunct}}}
  \bibliography{biblio}

\newtheorem{thm}{Theorem}[section]
\newtheorem{cor}[thm]{Corollary}
\newtheorem{prop}[thm]{Proposition}
\newtheorem{lem}[thm]{Lemma}

\theoremstyle{definition}
\newtheorem{defn}[thm]{Definition}
\theoremstyle{remark}

\theoremstyle{definition}

\theoremstyle{definition}
\newtheorem{ex}[thm]{Example}
\theoremstyle{definition}

\numberwithin{equation}{section}
\title{Finite isoresidual covers in strata of $k$-differentials}

\author{Dawei Chen} 
\address[Dawei Chen]{Department of Mathematics, Boston College, Chestnut Hill, MA 02467, USA}
\email{dawei.chen@bc.edu}

\author{Quentin Gendron}
\address[Quentin Gendron]{Instituto de Matem\'{a}ticas de la UNAM,
Ciudad Universitaria, Coyoacán, 04510, Tenochtitlán,
M\'{e}xico}
\email{gendron@matem.unam.mx}

\author{Miguel Prado} 
\address[Miguel Prado]{Goethe University, Frankfurt am Main, Hessen 60325, Germany}
\email{prado@math.uni-frankfurt.de}

\author{Guillaume Tahar}
\address[Guillaume Tahar]{Beijing Institute of Mathematical Sciences and Applications, Huairou District, Beijing, China}
\email{guillaume.tahar@bimsa.cn}

\date{\today}
\keywords{$k$-differentials, isoresidual fibration, resonance stratification, multi-scale compactification, cone spherical metric}

\begin{document}
\begin{abstract}
Consider the strata of primitive $k$-differentials on the Riemann sphere whose singularities, except for two, are poles of order divisible by $k$. The map that assigns to each $k$-differential the $k$-residues at these poles is a ramified cover of its image. Generalizing results known in the case of abelian differentials, we describe the ramification locus of this cover and provide a formula, involving the $k$-factorial function, for the cardinality of each fiber. We prove this formula using intersection calculations on the multi-scale compactification of the strata of $k$-differentials. In special cases, we also give alternative proofs using flat geometry. Finally, we present an application to cone spherical metrics with dihedral monodromy.
\end{abstract}
\maketitle
\setcounter{tocdepth}{1}
\tableofcontents

\section{Introduction}

A $k$-differential~$\zeta$ on a Riemann surface $X$ of genus $g$ is a section of $K^{\otimes k}$, where $K$ is the canonical bundle of~$X$. Thus, $k$-differentials encode intrinsic properties of $X$. Moreover, the integration of a $k$-th root of $\zeta\neq0$ induces a flat metric with conical singularities at the zeros and poles, where the transition maps (away from the singularities) are given by translations and rotations by angles that are multiples of $2\pi /k$. 
\par
The moduli space of $k$-differentials can be stratified according to the orders $m_{i}$ of the $n$ zeros and the $p$ poles of $\zeta$.  These spaces, called strata and denoted by $\komoduli(m_{1},\dots,m_{n+p})$, are complex orbifolds of dimension $2g-2+n+p$, as proved in \cite{BCGGM3}. In this work, we mark the zeros and poles: the strata parameterize tuples $(X,\zeta, z_{1},\dots,z_{n};p_{1},\dots,p_{p})$, where $z_{i}$ is a singularity of order $m_{i}>-k$ and $p_{i}$ of order $m_{n+i}$ of the $k$-differential~$\zeta$. These strata play a significant role in  understanding surface dynamics, as illustrated in \cite{ChenSurvey,AMTransSurf} for $k=1,2$ and in \cite{AAHR} for higher~$k$.
\par
From the perspective of enumerative geometry, there have been a number of fascinating recent results concerning $k$-differentials and their  applications. These include computing volumes and Euler characteristics of linear submanifolds in moduli spaces of flat surfaces, using cycle classes of strata of differentials to study double ramification cycles, and developing analogous structures, called $k$-leaky numbers, that arise in the Hurwitz counting problem of branched covers. We refer to \cite{SauvagetVolumes, NguyenVolumes, CMSEuler, SchmittKdiff, BHPSS, CMSLeaky} for these developments.
\smallskip
\par
As we recall in Section~\ref{sec:locmodel}, if a singularity of $\zeta$ is a flat-geometric zero (i.e., of order $>-k$), or if it is a pole whose order is not divisble by $k$, then its $k$-residue is always zero. On the other hand, a pole of $\zeta$ of order divisible by $k$ may have a non-trivial $k$-residue.
\par
A dimension count shows that for $k \geq 2$, the only strata whose dimension equals that of their space of configurations of $k$-residues are of the form $\Omega^{k}\mathcal{M}_{0}(a_{1},a_{2},-b_{1},\dots,-b_{p})$, where $b_{1},\dots,b_{p}\in k \mathbb{N}_{>0}$ and $a_{1},a_{2}\in\ZZ$ are coprime to $k$. In this case, we define the {\em isoresidual fibration} as
\begin{equation*}
 \Res\colon\komoduli[0](a_1, a_2, -b_1, \ldots, -b_p) \to \mathcal{R}^{k}_{p} : \zeta \mapsto (\Res_{p_{1}}\zeta,\dots,\Res_{p_{p}}\zeta)\,,
\end{equation*}
where the target space $\mathcal{R}^{k}_{p}$ is the complex vector space $\mathbb C^p$ that encodes the configurations of the $k$-residues at the poles $p_{1},\dots,p_{p}$ of order $-b_1,\ldots, -b_p$.
\par
Since we label each singularity, we assume throughout the paper that $a_1 > -k$. For $a_2$, there are two cases: either $a_2 > -k$, corresponding to a flat-geometric zero, or $a_2 < -k$, corresponding to a flat-geometric pole whose order is not divisible by $k$; in both cases, the $k$-residue is zero. 
\par 
In these cases, when $p\geq2$, the isoresidual fibration is a ramified cover of its image, which we call the {\em $k$-isoresidual cover}. The main goal of this paper is to compute the number of preimages above any configuration of residues. In particular, we determine the degree and the ramification locus of this cover. Remarkably, the degree formula involves the $k$-factorial function, defined by $$a!_{(k)}\coloneqq\prod\limits_{i=0}^{\lceil a/k \rceil}(a-ik)\,.$$ 
More generally, we define the following function.
\begin{defn}
Let $k\geq2$ and $a,m\in\NN_{>0}$. The {\em $(k,m)$-partial product of $a$} is defined as  
\begin{equation}\label{eq:fk}
 f_{k}(a,m)=\left \{
   \begin{array}{l r}
     \frac{1}{a+k}   & \text{ if } m=1\,,\\
      1   & \text{ if } m=2\,, \\
      \prod\limits_{0 \leq j \leq m-3} (a-kj)  & \text{ if } m\geq3\,.
   \end{array}
   \right .
\end{equation}
\end{defn}
\par 
Note that when $m\leq\left[\frac{a}{k}\right]+3$, the value of $f_{k}(a,m)$ is positive, whereas for $m>\left[\frac{a}{k}\right]+3$, the value of $f_{k}(a,m)$ alternates in sign as $m$ increases. 
Using this function, we can determine the degree of the $k$-isoresidual cover as follows.
\begin{thm}\label{thm:main}
Let $k\geq2$ and $\mu=(a_1,a_2,-b_1,\ldots,-b_p)$ be a partition of $-2k$ such that $b_{1},\dots,b_{p}$ are positive integer multiples of $k$, while $a_{1}$ and $a_{2}$ are coprime to $k$.
The isoresidual map from the stratum  $\Omega^{k}\mathcal{M}_{0}(\mu)$ of (marked) $k$-differentials on $\mathbb{CP}^{1}$ is a ramified cover of~$\mathbb{C}^{p}$ of degree
 \begin{equation}\label{eq:degre1}
  d_{k}(\mu)\coloneqq\sum_{c_{1,I}>0} c_{1,I} \cdot f_{k}(a_{1},|I|+1)\cdot f_{k}(a_{2},|I^c|+1)\,,
 \end{equation}
where for any $I\subset \{1,\ldots,p\}$
\begin{equation}\label{eq:c1I}
 c_{1,I}=a_1-\sum\limits_{i\in I}b_i+k \,.
\end{equation}

\end{thm}
The above formula takes a particularly simple form when all the poles have order $-k$.
\begin{cor}\label{cor:polesk}
 For the strata $\Omega^{k}\mathcal{M}_{0}(a_{1},a_{2},\rec[-k][p])$, the degree of the isoresidual cover is given by 
 \begin{equation*}
  \binom{p-1}{\lceil a_{1}/k \rceil} \cdot
a_{1}!_{(k)} \cdot a_{2}!_{(k)}\,.
 \end{equation*}
\end{cor}
\smallskip
\par
Now, to describe the ramification locus of the isoresidual fibration, we introduce the following polynomial expressions in the $k$-residues. Let $[R_1:\cdots:R_p]$ be a tuple of $k$-residues, defined up to simultaneous scaling by $\CC^{\ast}$. For any subset $I = \{i_1,\ldots, i_d\}$ of $\{1,\ldots, p\}$, we define the polynomial
\begin{equation}\label{eq:polyanul}
    P(R_{i_1},\ldots,R_{i_d})\coloneqq \prod_{\left\{(r_{i_1},\ldots,r_{i_d})\colon 
    r_{i_j}^k = R_{i_j}\right\}}(r_{i_1}+\cdots+r_{i_d})\,,
\end{equation}
where the product
runs over all tuples of $k$-th roots of the $k$-residues. This is a homogeneous polynomial of degree $k^{d-1}$ in the ring $\mathbb{Z}[R_1,\ldots,R_p]$ (see Lemma~\ref{lem:PSymmetric}).
\par
The subset of indices $I$ for which the above polynomial vanishes is called a \textit{resonant subset}, and the locus in the space of $k$-residue tuples for which a resonant subset exists is called the {\em resonant locus}. Note that the resonant locus is a union of hypersurfaces, each of which generically corresponds to exactly one resonant subset.
\par
One difference compared to the case of abelian differentials is that there is no residue theorem for $k$-differentials when  $k\geq2$. Hence, a subset $I$ being resonant does not imply that $I^c$ is resonant; in fact, we allow $I=\{1,\ldots,p\}$. When a resonant equation occurs, some sum of $k$-th roots of the residues equals zero. Nevertheless, there may be more than one such sum that vanishes. This motivates the following definition.
\begin{defn}\label{def:abeliannb}
Given a $k$-residue tuple $\calR = [R_1:\cdots:R_p]$ and a subset of indices $I$, the {\em abelian number} of $I$ with respect to $\calR$ is defined as 
\begin{equation}\label{eq:abeliannb}
    \Ab_{\mathcal{R}}(I)  =  \#\left\{(r_i)_{i\in I} \ \Bigg| \  \sum_{i\in I}r_i=0 \text{ and }  \forall i\in I, \  r_i^k=R_i  \right\}\Bigg/\mathbb{C}^*\, .
\end{equation}
\end{defn}
\par
With these preparations, we can determine the number of $k$-differentials whose residues satisfy a unique resonant equation.
\begin{thm}\label{thm:one-resonant}
Let $k\geq2$ and $\mu=(a_1,a_2,-b_1,\ldots,-b_p)$ be a partition of $-2k$ such that $b_{1},\dots,b_{p}$ are positive integer multiples of $k$ and $a_{1}$ and $a_{2}$ are coprime to~$k$. Let  $I\subseteq\{1,\ldots,p\}$ be a subset of $\{1,\dots,p\}$ and denote $B_I= \sum_{i\in I}b_{i}$.  The cardinality of the isoresidual fiber over a residue tuple $\mathcal{R}=[R_1:\cdots:R_p]$ which satisfies exactly the resonance equation defined by ~$I$ is the piecewise polynomial of degree $p-1$ given by 
    \begin{equation*}
     \left\{
   \begin{array}{l r}
    d_{k}(\mu)- f_{I}\cdot \Ab_{\mathcal{R}}(I)  & \text{ if } I = \{1,\ldots,p\}\,,\\
       d_{k}(\mu)- \left[c_{1,I_{>0}}d_k(\mu_{I,\emptyset})+c_{2,I_{>0}}d_k(\mu_{\emptyset, I})\right]f_{I} \cdot \Ab_{\mathcal{R}}(I)   & \text{ if } I\subsetneq \{1,\ldots,p\}\,,
   \end{array}
\right.
    \end{equation*}
where $f_I=f_k(B_I-k,|I|+1)$, $\mu_{\emptyset,I}=(a_1,c_{2,I}-k,\{-b_i\}_{i\in I^c})$, $\mu_{I,\emptyset}=(c_{1,I}-k,a_2,\{-b_i\}_{i\in I^c})$, and  $c_{i,I_{>0}} = \max\{0,c_{i,I}\}$.
\end{thm}
When there are an arbitrary number of resonant subsets, the set of indices can be partitioned in several ways as $\{1,\ldots,p\} = J_0 \sqcup J_1 \sqcup \cdots\sqcup J_s$, where $J_1,\ldots,J_s$ are resonant subsets, and the subset $J_0$ may be non-resonant and possibly empty. For example, the degree of the generic $k$-isoresidual fiber corresponds to the case $J_0=\{1,\ldots,p\}$, while for a $k$-isoresidual fiber with a single resonant subset $I$, we subtract from the degree of the generic fiber a term corresponding to $J_1=I$ and $J_0= I^c$. In general, the degree of the $k$-isoresidual fiber over an arbitrarily given $k$-residue tuple involves a contribution from each such partition, described as follows. 
\par 
\begin{thm}\label{thm:iso-arbitrary}
Let $k\geq2$ and $\mu=(a_1,a_2,-b_1,\ldots,-b_p)$ be a partition of $-2k$ such that $b_{1},\dots,b_{p}$ are positive integer multiples of $k$, while $a_{1}$ and $a_{2}$ are coprime to~$k$. Given an arbitrary tuple of $k$-residues $\mathcal{R}=[R_1:\cdots:R_p]$, the cardinality of the $k$-isoresidual fiber at $\mathcal{R}$ is   
    \begin{equation*}
        \sum_{J_0\sqcup J_1\sqcup \cdots \sqcup J_s}(-1)^sG_k(J_0;J_1,\ldots,J_s)\prod_{j=1}^{s}f_{J_j}\Ab_{\mathcal{R}}(J_j)\,,
    \end{equation*}
    where the sum ranges over all possible partitions   $J_0; J_1,\ldots,J_s$ such that each $J_i$ for $1\leq i \leq s$ is a resonant subset with respect to $\calR$, and  
    \begin{equation*}
      G_k(J_0;J_1,\ldots,J_s)  =  \begin{cases} 
      \sum_{d_{1,I}>0 } d_{1,I}(a_1+k)^{|I|-1}(a_2+k)^{|I^c|-1} & \text{if} \ J_0=\emptyset\,, \\
      
      \sum_{\substack{d_{1,I}>0 \\ d_{2,I^c}>0}} d_{1,I}d_{2,I^c}d_k(\mu_{I,I^c})(a_1+k)^{|I|-1}(a_2+k)^{|I^c|-1} & \text{if} \ J_0\not=\emptyset\,,
   \end{cases} 
    \end{equation*}
where for    $i=1,2$ and $I\subseteq\{1,\ldots,s\}$ 
\begin{equation*}
 d_{i,I}  =  c_{i, \ \sqcup_{j\in I} J_{j}}  \text{ and }\mu_{I,I^c} =  (d_{1,I}-k,d_{2,I^c}-k,\{-b_i\}_{i\in J_0}) \,.
\end{equation*}
\end{thm}
\par 
Note that Theorem~\ref{thm:iso-arbitrary} could theoretically be used to recover  the classification of empty $k$-isoresidual fibers for such strata, as established in \cite[Theorems 1.6 and 1.9]{getaquad} (for quadratic differentials) and \cite[Theorem 1.6]{GT-k-residue} (for $k$-differentials with $k \geq 3$). However this produces delicate combinatorial identities involving the cardinality of the isoresidual fiber over each intersection of resonance hypersurfaces (see the examples at the end of Section~\ref{sec:ex}).
\par 
Finally, we remark that the above results generalize the case $k=1$, which was studied in \cite{GeTaIso} using flat geometry and later completed in \cite{ChPrIso} using intersection theory. Applications and alternative perspectives on this problem can be found in \cite{Sugiyama,BuRoCoun, OberWRS}. It seems to be unknown if these applications and perspectives can be extended to the context of $k\geq2$.
\smallskip
\par
\paragraph{\bf Organization of the paper:}
\begin{itemize}
    \item In Section~\ref{sec:FlatMetric}, we review the background on $k$-differentials, the local invariants of their singularities, their strata, and their relation to flat geometry. In particular, we introduce the residue map and the isoresidual fibration.
    \item In Section~\ref{sec:Resonance}, we introduce the resonance stratification of the $k$-residue space and show that the number of elements in isoresidual fibers is constant over the complement of the resonance locus.
    \item In Section~\ref{sec:multi-scale}, we describe the multi-scale compactification of strata of $k$-differentials.
    \item In Section~\ref{sec:Intersection}, we use intersection theory to compute the degree of the residue map above all the residue tuples, proving Theorems~\ref{thm:main},~\ref{thm:one-resonant}, and~\ref{thm:iso-arbitrary}. 
    \item In Section~\ref{sec:Flat}, we provide alternative flat-geometric arguments to prove  Corollary~\ref{cor:polesk}  and give various examples of computations.
    \item In Section~\ref{sec:Spherical}, we deduce, from Theorem~\ref{thm:main}, a counting result for a special class of spherical metrics, stated in Theorem~\ref{thm:Spherical}.
\end{itemize}

\paragraph{\bf Acknowledgements.} Research of D.C. is supported by National Science Foundation grant DMS-2301030, Simons Travel Support for Mathematicians, and a Simons Fellowship. Research of Q.G. is supported by the grant PAAPIT UNAM-DFG DA100124 ``Conectividad y conectividad simple de los estratos.'' Research of M.P. is supported by the DFG-UNAM project MO 1884/3-1 and the Collaborative Research Centre TRR 326 ``Geometry and Arithmetic of Uniformized Structures.'' Research of G.T. is supported by the Beijing Natural Science Foundation IS23005 and the French National Research Agency under the project TIGerS (ANR-24-CE40-3604).

\section{$k$-differentials}\label{sec:FlatMetric}

In this section, we review some general properties of $k$-differentials and their associated flat structures. The results collected here were proved in \cite{BCGGM3}.

\subsection{$k$-differentials and local models of singularities}\label{sec:locmodel}

A $k$-differential on a Riemann surface $X$ is a (holomorphic or meromorphic) section  that we assume to be not identically zero of the $k$-th tensor power of its canonical bundle. The degree of the associated $k$-canonical divisor is $k(2g-2)$.

Locally, a $k$-differential posses two invariants: its order and its $k$-residue. More precisely, given a $k$-differential $\zeta$, there exists a neighborhood of any point~$P$ and a biholomorphic change of coordinates such that~$\zeta$ takes the form:
\begin{equation}\label{eq:standard_coordinates}
    \begin{cases}
      z^m\, (dz)^{k} &\text{if $m> -k$ or $k\nmid m$,}\\
      \left(\frac{r}{z}\right)^{k}(dz)^{k} &\text{if $m = -k$,}\\
        \left(z^{m/k} + \frac{t}{z}\right)^{k}(dz)^{k} &\text{if $m < -k$ and $k\mid m$,}
    \end{cases}
\end{equation}
where $r \in \mathbb{C}^{\ast}$ and $t \in \mathbb{C}$. The integer $m$ is the {\em order} of $\zeta$ at the point~$P$.  The constants $r$ and $t$ are defined up to multiplication by a $k$-th root of unity. We define the {\em $k$-residue} $\Res_{P}(\zeta)$ of $\zeta$ at $P$ as~$r^{k}$ in the second case,  $t^{k}$ in the third case, and zero otherwise. In particular, the $k$-residue can be nonzero only for poles whose order is a  multiple of $k$.

Note that for every $k\geq2$ there is no residue theorem; that is, the sum of the $k$-residues of the poles on a compact Riemann surface need not vanish.

\subsection{Strata of $k$-differentials}\label{sec:strata}

Let  $\mu$ be a partition of $k(2g-2)$. We define the {\em stratum $\komoduli(\mu)$ of $k$-differentials of type $\mu$} to be the space of {\em primitive} $k$-differentials with marked zeros and marked poles of orders  prescribed by~$\mu$. Recall that a $k$-differential is primitive if it is not the global power of a lower-order differential. Except for a few cases (see \cite{Diaz-Marin,getaquad,GT-k-residue}), these strata are nonempty and form smooth complex orbifolds of dimension $2g-2+|\mu|$.
\par
In genus zero, it is easy to describe the partitions $\mu$ for which the strata of primitive $k$-differentials $\komoduli[0](\mu)$ are nonempty: they are precisely those $\mu$ for which the greatest common divisor of its parts is coprime to $k$.
\medskip
\par
Suppose that $\mu=(a_{1},\dots,a_{n},-b_{1},\dots,-b_{p},-c_{1},\dots, -c_{r})$ with $a_{i}>-k$, each $b_{i}$ divisible by $k$, and each $c_{i}$ not divisible by $k$. We define the \textit{residual space} $\mathcal{R}_{p}^{k}$ to be the complex vector space~$\CC^{p}$. Each stratum $\komoduli(\mu)$ is endowed with a \textit{residual map}
\begin{equation}\label{eq:resmap}
 \Res\colon\komoduli(\mu) \to \mathcal{R}^{k}_{p} : \zeta \mapsto (\Res_{p_{1}}\zeta,\dots,\Res_{p_{p}}\zeta)
\end{equation}
which assigns to each $k$-differential $\zeta$ the sequence of its $k$-residues at the poles $p_{i}$ of orders divisible by $k$. This map defines the \textit{isoresidual fibration} on the stratum $\komoduli(\mu)$.
\medskip
\par
In the case of abelian differentials (i.e., $k=1$), 
it is well known (see, for example, \cite[Section 3.3]{AMTransSurf}) that the strata admit local coordinates, called period coordinates, given by integration of the differential along a basis of the first homology of the surface without the poles, relative to the zeros of the differential. A similar description holds for strata of $k$-differentials with $k\geq2$ (see \cite[Corollary 2.3]{BCGGM3}). Roughly speaking, the coordinates are given by an eigenspace for the cyclic action on the canonical cover of the $k$-differentials.

\subsection{$(1/k)$-translation structures}\label{sec:transsurf}

Recall that to each $k$-differential $\zeta$ there is an associated well-defined $(1/k)$-translation structure. It is obtained by integrating a $k$-th root of $\zeta$, and is well defined up to a  rotation by an angle of $2\pi/k$. This yields a structure consisting of translations together with rotations by angles that are multiples of $2\pi/k$, with singularities whose cone angles are also multiples of $2\pi/k$.

Note that a zero of order $a_{i}\geq -k+1$ corresponds to a conical singularity with angle $\frac{a_{i}+k}{k}2\pi$. A {\em saddle connection} is a geodesic arc joining two conical singularities.

\section{Resonance stratification}\label{sec:Resonance}

We review the stratification of the residue space introduced for  quadratic differentials in \cite[Section 2.5]{getaquad} and extended to all $k\geq3$ in \cite[Section 2.4]{GT-k-residue}.

\subsection{Periods of saddle connections}\label{sub:Periods}

Given a $k$-differential $\zeta$ in $\Omega^{k}\mathcal{M}_{0}(a_{1},a_{2},-b_{1},\dots,-b_{p})$ with $k$-residues  $(R_{1},\dots,R_{p})$, the lengths of its saddle connections are proportional to the norms of partial sums of $k$-th roots of the~$R_{i}$.
\par
Indeed, any closed saddle connection $\gamma$ decomposes $\mathbb{CP}^{1}$ into two connected components, one of which is a translation surface bounded by $\gamma$. Applying the residue theorem to this component shows that the length of $\gamma$ coincides with the modulus of the sum of the $1$-residues at the poles contained in this component. Since these $1$-residues are $k$-th roots of the corresponding $k$-residues, the length of $\gamma$ is proportional to a partial sum of the $k$-th roots of $R_{1},\dots,R_{p}$.
\par
The case of saddle connections joining two distinct conical singularities is slightly more involved. Since there are at most two singularities whose order is not a multiple of $k$, cutting along such a saddle connection $\gamma$ yields a translation surface with two boundary saddle connections of period $z$ and $-\epsilon z$, where $\epsilon$ is a $k$-th root of unity. It then follows from the residue theorem that $(1-\epsilon)z+\sum\limits_{j=1}^{p} r_{j} =0$, where $r_{1},\dots,r_{p}$ are again $k$-th roots of the $k$-residues $R_{1},\dots,R_{p}$.
\par
Hence, the degeneration of any saddle connection corresponds to the vanishing of certain sums of $k$-th roots of the $k$-residues. We will encode these conditions in the definition of the \textit{resonance locus} below. 

\subsection{Resonance stratification of the residual space}\label{sub:ResStrat}

Let $W_{k}$ be $\lbrace{ 0\rbrace} \cup \lbrace{ e^{\frac{2\ell i\pi}{k}}~|~\ell \in \mathbb{Z}/k\mathbb{Z} \rbrace}$. In $\mathbb{C}^{p}$, a \textit{resonance hyperplane} is the set of points $(r_{1},\dots,r_{p})$ satisfying a nontrivial equation of the form
$$\sum\limits_{j=1}^{p} w_{j}r_{j}=0,\quad w_{j}\in W_{k}\,.$$
The collection $H_{k,p}$ of resonance hyperplanes defines a complex hyperplane arrangement in $\mathbb{C}^{p}$.

\begin{defn}\label{defn:stratres}
Let $\calR=(R_{1},\dots,R_{p})\in \mathcal{R}^{k}_{p}$ be a tuple of $k$-residues. We define the {\em set of resonance hyperplanes} of $\calR$ as $H_{k,p}(\calR) \subset H_{k,p}$, where each resonance hyperplane in $H_{k,p}(\calR)$ contains a tuple  $(r_{1},\dots,r_{p})$ of $k$-th roots of the corresponding entries of $\calR$. 
\par
We further define the \textit{resonance locus} of $\mathcal{R}^{k}_{p}$ as the image of the union of resonance hyperplanes under the map 
\begin{equation}\label{eq:projection}
 \CC^{p}\to\mathcal{R}^{k}_{p}:(r_{1},\dots,r_{p}) \mapsto (r_{1}^{k},\dots,r_{p}^{k})\,.
\end{equation}
Finally, we define the \textit{resonance stratification} as follows: two $k$-residue tuples $\calR$ and $\calR'$ belong to the same resonance stratum if and only if $H_{k,p}(\calR)=H_{k,p}(\calR')$. 
\end{defn}
\par
We remark that each stratum of the resonance stratification is an algebraic subset of $\mathcal{R}^{k}_{p}$.
\smallskip
\par 
Next, we give a description of the resonance locus by using the polynomial~$P$ introduced in Equation~\eqref{eq:polyanul}. First let us prove that this is in fact a polynomial.
\begin{lem}\label{lem:PSymmetric}
For any subset $I = \{i_1,\ldots, i_d\}$ of $\{1,\ldots, p\}$, the polynomial
\begin{equation*}
    P(R_{i_1},\ldots,R_{i_d})\coloneqq \prod_{\left\{(r_{i_1},\ldots,r_{i_d})\colon 
    r_{i_j}^k = R_{i_j}\right\}}(r_{i_1}+\cdots+r_{i_d})\,,
\end{equation*}
defined in Equation~\eqref{eq:polyanul} is a homogeneous polynomial of degree $k^{d-1}$ in the ring $\mathbb{Z}[R_1,\ldots,R_p]$.
\end{lem}

\begin{proof}
Given an arbitrary choice $(r_{i_{1}},\dots,r_{i_{d}})$ of $k$-th roots of the $k$-residues $(R_{i_1},\ldots,R_{i_d})$, the polynomial $P$ is homogeneous of degree $k^{d}$ in $\mathbb{Z}[r_{i_{1}},\dots,r_{i_{d}}]$. By construction, $P$ is  symmetric in $r_{i_{1}},\dots,r_{i_{d}}$. Then, for any variable $r_{i_{s}}$, we have $P=\sum_{\ell} r_{i_{s}}^{\ell}Q_{\ell} (r_{i_{1}},\dots,\hat{r}_{i_{s}},\dots,r_{i_{d}})$,  where $Q_{\ell}$ is a polynomial of degree $d-\ell$ in the remaining $d-1$  variables. Since $P$ is invariant under the substitution $r_{i_{s}}\mapsto\zeta r_{i_{s}}$, where $\zeta$ is a $k$-th root of unity, it follows that $Q_{\ell}=0$ unless $\ell \in k\mathbb{N}$. This implies that $P$ can be written as a polynomial in the variables $r_{i_{1}}^{k},\dots,r_{i_{d}}^{k}$. Therefore, $P$ is a homogeneous polynomial of degree $k^{d-1}$ in $\mathbb{Z}[R_1,\ldots,R_p]$.
\end{proof}

Observe that  $P(R_{i_1},\ldots,R_{i_d})$ vanishes if and only if $(R_{1},\dots,R_{p})$ belongs to the resonance locus. This observation motivates the following definition, which is used to rule out certain degenerate cases in the proof of Theorem~\ref{thm:main}.
\begin{defn}
    For $k>1$, a tuple of $k$-residues $\mathcal{R}=(R_1, \dots, R_p)$ is called \textit{general} if there is no subset $I\subset \{1,\ldots,p\}$ such that $P(\{R_i\}_{i\in I})=0$.
\end{defn}

\subsection{The residual systole and flat continuation}

In each resonance stratum, we define a \textit{residual systole} to control deformations of $k$-differentials.

\begin{defn}\label{defn:systole}
Let $\calR=(R_{1},\dots,R_{p})$ be a tuple of $k$-residues. We define the \textit{residual systole} $\sigma(\calR)$ for $\calR$ as 
 \[ \sigma(R) = \min \left\{ \left| \sum_{i\in I} r_{i} \right| : I\subset \lbrace1,\dots,p \rbrace, r_{i}^{k}=R_{i} \text{ and } \sum_{i\in I} r_{i} \neq 0 \right\}\,.\]
\end{defn}

\begin{prop}\label{prop:systole}
The residual systole is a continuous function on every resonance stratum.
\end{prop}

\begin{proof}
There are finitely many weighted sums of $k$-th roots, and each of them varies continuously as the $k$-residues vary.
\end{proof}

We now introduce a flat-geometric variant of analytic continuation. 

\begin{cor}\label{cor:defplate}
Consider a stratum $\Omega^{k}\mathcal{M}_{0}(a_{1},a_{2},-b_{1},\dots,-b_{p})$ of $k$-differentials. Let $\mathcal{S}$ be a stratum of $\mathbb{C}^{p}$ equipped with the resonance stratification. Then every isoresidual fiber over $\mathcal{S}$ contains the same number of elements.
\end{cor}

\begin{proof}
Let $\zeta$ be a $k$-differential whose configuration of $k$-residues is $\calR=(R_{1},\dots,R_{p}) \in \mathcal{S}$. In each chart of $\Omega^{k}\mathcal{M}_{0}(a_{1},a_{2},-b_{1},\dots,-b_{p})$, every saddle connection has a length proportional to a partial sum of $k$-th roots $r_{1},\dots,r_{p}$ of $R_{1},\dots,R_{p}$. These lengths are bounded below by the residual systole $\sigma(\calR)$ (see Definition~\ref{defn:systole}).
\par
Since the residual systole varies continuously in $\mathcal{S}$, there exists a neighborhood $V$ of $\calR$ in $\mathcal{S}$ where no saddle connection of $\zeta$ can degenerate as the flat surface $(\mathbb{CP}^{1},\zeta)$ is deformed. It follows that the number of $k$-differentials in the isoresidual fiber is locally constant on each resonance stratum.
\end{proof}

In the isoresidual cover, a {\em generic fiber} is an isoresidual fiber lying over a point in the generic stratum of the resonance stratification; that is, over a point in the complement of the resonance locus.

\begin{cor}\label{cor:degreeGen}
Given a stratum $\Omega^{k}\mathcal{M}_{0}(a_{1},a_{2},-b_{1},\dots,-b_{p})$ of $k$-differentials, the degree of the isoresidual cover equals the number of elements in each generic fiber. 
\end{cor}

\section{The multi-scale compactification of isoresidual loci}\label{sec:multi-scale}

In this section, we recall the basics of the multi-scale compactification for strata of $k$-differentials, as studied in \cite{CoMoZaArea} and based on the works \cite{BCGGM3,BCGGM2}. This framework allows us to describe the closure of loci of $k$-differentials whose $k$-th roots of $k$-residues satisfy certain linear relations. Moreover, we discuss properties of multi-scale $k$-differentials contained in the closure of the isoresidual loci within this compactification. Since Theorem~\ref{thm:closureMSD} in this section applies to any genus, we denote by $\komoduli(\mu)$ a stratum of $k$-differentials of genus $g\geq0$ with signature $\mu$ when the results hold in these general cases.
\par
\subsection{The multi-scale compactification of strata}\label{sub:multi-scale}
\par
The multi-scale compactification $\MS(\mu)$ of the projectivized strata $\PP\komoduli(\mu) \coloneqq \komoduli(\mu) / \mathbb C^*$ is constructed in \cite{CoMoZaArea}. Here, we follow \cite[Section~2.1]{chge} to give the reader an introduction to {\em multi-scale $k$-differentials}. 
\begin{def}  \label{def:mskd}
A {\em multi-scale $k$-differential $(X,\bfz,\zeta,\preccurlyeq,\sigma)$ of type $\mu$} (usually written simply as $(X,\zeta,\sigma)$) consists of: 
\begin{itemize}
  \item[(i)] a stable pointed curve $(X,\bfz)$ with an enhanced level structure $\preccurlyeq$ on the dual graph~$\Gamma$ of~$X$;
  \item[(ii)] a twisted $k$-differential $(X,\bfz,\zeta)$ of type~$\mu$ together with a $k$-prong-matching $\sigma$ compatible with the enhanced level structure. 
\end{itemize}
\end{def}

In this definition, $(X,\bfz)$ is a stable pointed nodal Riemann surface of genus $g$, and $\zeta$ consists of a non-identically-zero $k$-differential $\zeta_{i}$ on each irreducible component $X_{i}$ of $X$. The total order $\preccurlyeq$ compares any two irreducible components of $X$ and encodes information about the vanishing rates of  differentials from nearby smooth surfaces as they degenerate to the sub-surfaces $X_i$. This order induces a level structure on the dual graph of $X$, which is referred to as a {\em level graph}.

Recall that the sum of the orders of the $\zeta_{i}$ at the two branches of every node is equal to $-2k$, and in the case of poles of order $-k$, the two $k$-residues satisfy $R_{1}+ (-1)^{k}R_{2}=0$.  If a node has two poles of order $k$ at its branches, the corresponding edge in the dual graph is called {\em horizontal}; otherwise, it is called {\em vertical}. At a vertical edge, if the multi-scale differential has a zero of order $n\geq0$ at the upper nodal point, the number $n+k$ is called the {\em prong number} of the node. This number gives the count of horizontal directions (up to a $k$-th root of unity) at the node.

Next, we discuss the $k$-prong-matching $\sigma$ and the global $k$-residue condition, which differs from the abelian case for $k\geq2$. Given a vertical edge~$e$ of the enhanced level graph~$\Gamma$, a {\em (local) $k$-prong-matching} $\sigma_e$ is defined as a cyclic, order-reversing bijection between the $k$-prongs at the upper and lower ends of~$e$. A {\em (global) $k$-prong-matching} is a collection
$\sigma = (\sigma_e)_{e \in E(\Gamma)^v}$ of local $k$-prong-matchings at every vertical edge. 

Consider a level $L$ and a component $Y$ of the part $\Gamma_{>L}$ of sub-surfaces lying strictly above $L$ in~$\Gamma$. We say that the restriction of the multi-scale $k$-differential $(X,\zeta)$ to $Y$ is of {\em abelian type} if this restriction is the $k$-th power of a multi-scale abelian differential (see \cite[Section~2.1]{chge} for  details). If $Y$ is not of abelian type, or if it contains a pole of $\zeta$, then there are no conditions on the residues at the poles connecting $Y$ with the rest of $X$. If $Y$ is of abelian type and does not contain a pole, it must satisfy the (usual) {\em global $k$-residue condition}:
\begin{itemize}
 \item[{\bf $k$-GRC}:]  $P(\Res_{e_{i}}(\zeta)) = 0$\,,
\end{itemize}
where $P$ is the polynomial defined in Equation~\eqref{eq:polyanul}.

\subsection{The closure of loci of $k$-differentials with linear residue conditions}\label{sub:closure}

One can define a similar residual map from the multi-scale compactification $\MS(\mu)$ to the projective residue space $\mathbb{P} \mathcal{R}^{k}_{p} = (\mathcal{R}^{k}_{p} \setminus \{ 0 \}) / \mathbb{C}^{*}$. In general, this map is only a rational map; for example, it is undefined on the locus of residueless differentials.
\par
Following \cite[Section 3]{CGPT1}, we now describe the closure of isoresidual fibers in the moduli space of multi-scale $k$-differentials. This discussion applies to arbitrary genus and partitions~$\mu$. Moreover, we work in the more general setting where a linear subspace of the $k$-th roots of the $k$-residues is fixed.
\par
Consider a stratum $\PP\komoduli(\mu)$ of meromorphic $k$-differentials of genus $g$ with signature~$\mu$. Let~$\Lambda$ be a linear subspace of the cover $\CC^{p}$ of the residual space $\mathcal{R}_{p}^{k}$ given in Equation~\eqref{eq:projection}.  Consider the subspace $\mathcal{F}_{\Lambda}$ consisting of $k$-differentials whose $k$-residues lie in the projection of~$\Lambda$. We define ${\Lambda}^{\vee}\subset (\CC^{p})^{\vee}$ to be the vector space of homogeneous linear equations satisfied by all residue tuples in $\Lambda$.
\par 
We aim to characterize when a multi-scale $k$-differential $(X,\zeta,\sigma)$ lies in the closure of~$\mathcal F_\Lambda$. To this end, we introduce the {\em generalized global $k$-residue condition} imposed on $(X,\zeta,\sigma)$ by~$\Lambda^{\vee}$, which we  denote by $\mathcal{E}_{\Lambda}^{k}$-GRC. Let $q_1,\ldots, q_p$ be the marked poles in $\Gamma$.  For each $q_i$, add a new vertex at level $\infty$ with a marked pole $q'_i$, and replace the original $q_i$ with an edge connecting it to this new vertex. The new vertex can be regarded as a semistable rational component carrying a pole at $q'_i$ of the same order as $q_i$. Denote by $\Gamma'$ the resulting level graph. For every finite level $L$ of $\Gamma'$, let $Y_1,\ldots, Y_s$ denote the connected components of $\Gamma'_{>L}$.
\par 
If $q_{i}$ does not appear in any equation of $\Lambda^{\vee}$, we say that it is a {\em free} pole.
The following conditions then hold:
\begin{itemize}
 \item[{\bf $\mathcal{E}_{\Lambda}^{k}$-GRC}:]
 For every $Y_i$ of abelian type, one of the following conditions holds: 
 \begin{enumerate}
  \item $Y_{i}$ contains a free pole. 
  \item For all equations $f\in \Lambda^{\vee}$ that can be written in the form
$$f = \sum_{i=1}^s a_i \left(\sum_{q'_j\in Y_i} r_{q'_j}\right)\,,$$
we require that there exist roots $r_{e_{j}}$ of the residues $R_{e_{j}}$ such that
$$\sum_{i=1}^s a_i \left(\sum_{e_j \in Y_{i, L}} r_{e_j} (\zeta)\right) = 0\,,$$
where the inner summation ranges over the lower endpoints of the edges in $Y_i$ that connect to vertices at level $L$.
 \end{enumerate}
 \end{itemize}
\par
We remark that, up to linear combinations, there are  only finitely many non-trivial conditions imposed by the $\mathcal{E}_{\Lambda}^{k}$-GRC. To see this, note that if two independent relations in condition~(2) involve the same subset of $Y_i$ with nonzero coefficients, their combination can produce another relation involving a smaller subset of $Y_i$. Repeating this process, one can reduce to a finite set of relations that form an echelon form at each level.
\par
In what follows, we discuss examples to illustrate the above description and the $\mathcal{E}_{\Lambda}^{k}$-GRC.
\begin{ex}
Consider a stratum parameterizing $k$-differentials of genus $g\geq 1$ with four poles, all of order divisible by~$k$. We impose the residue relation $r_{1}+3r_{2}+3r_{3}=0$.
Let $(X,\zeta,\sigma)$ be a multi-scale $k$-differential  in the boundary of this isoresidual locus, with level graph $\Gamma$ shown on the left of Figure~\ref{fig:exMSC} and the associated level graph~$\Gamma'$ on the right, where the marked poles of $\Gamma$ become the edges connecting to level~$\infty$ in~$\Gamma'$. Moreover, suppose that the restriction of~$\zeta$ to each component is the $k$-th power of an abelian differential.
\begin{figure}[ht]
\begin{tikzpicture}[scale=1]
\fill (0,2) coordinate (b1) circle (2pt);

\coordinate (a1) at (0,0);\fill (a1) circle (2pt);
\coordinate (a2) at (-.8,2);\fill (a2) circle (2pt);
\coordinate (a4) at (1,1);\fill (a4) circle (2pt);
\coordinate (a3) at (0,1);\fill (a3) circle (2pt);
\draw (a1) ..controls (-.6,1)  ..node[left]{$e_{1}$} (a2);
\draw (a1) --node[right]{$e_{5}$} (a3);
\draw (a1) ..controls (.8,.5)  ..node[right]{$e_{4}$} (a4);
\draw (a2) --++ (90:.3) node[above] {$q_{1}$};
\draw (b1) --++ (50:.3)node[above] {$q_{3}$};
\draw (b1) --++ (130:.3)node[above] {$q_{2}$};
\draw (a4) --++ (90:.3)node[above] {$q_{4}$};

 \draw (a3) ..controls (.2,1.5)  ..node[right] {$e_{3}$}  (b1);
 \draw (a3) ..controls (-.2,1.5)  ..node[left] {$e_{2}$}  (b1);

 \begin{scope}[xshift=5cm]
  \fill (0,2) coordinate (b1) circle (2pt);

\coordinate (a1) at (0,0);\fill (a1) circle (2pt);
\coordinate (a2) at (-.8,2);\fill (a2) circle (2pt);
\coordinate (a4) at (1,1);\fill (a4) circle (2pt);
\coordinate (a3) at (0,1);\fill (a3) circle (2pt);
\draw (a1) ..controls (-.5,1)  .. (a2);
\draw (a1) -- (a3);
\draw (a1) ..controls (.8,.5)  .. (a4);

 \draw[postaction={decorate}] (a3) ..controls (.2,1.5)  ..  (b1);
 \draw[postaction={decorate}] (a3) ..controls (-.2,1.5)  ..  (b1);

 \coordinate (c1) at (-.8,3);
\coordinate (c2) at (-.3,3);
\coordinate (c3) at (.3,3);
\coordinate (c4) at (1,3);
 \draw (a2) -- (c1);
\draw (b1) -- (c2);
\draw (b1) -- (c3);
\draw (a4) -- (c4);

\filldraw[fill=white] (c1) circle (2pt);
\filldraw[fill=white] (c2) circle (2pt);
\filldraw[fill=white] (c3) circle (2pt);
\filldraw[fill=white] (c4) circle (2pt);
 \end{scope}

\end{tikzpicture}
\caption{The level graphs $\Gamma$ and $\Gamma'$,  illustrating the $\mathcal E_{\Lambda}^{k}$-GRC. Vertices at level $\infty$ are shown in white.}\label{fig:exMSC}
\end{figure}
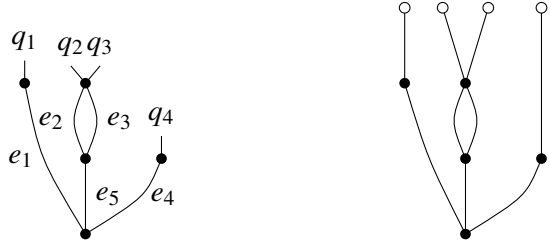
\par
The $\mathcal E_{\Lambda}^{k}$-GRC at level zero implies that the imposed residue relation still holds at level zero. However, this gives nothing beyond the residue theorem, yielding $r_1 = 0$ and $r_{2}+r_{3} = 0$.

Next, denote by $e_{i}$, for $i=1,2,3,4,5$, the residues of~$\zeta$ at the lower endpoints of the edges of $\Gamma$ with the same labels. Note that the $\mathcal E_{\Lambda}^{k}$-GRC at level $-1$ provides no additional information beyond the residue theorem. 

Now consider level $-2$, which presents two cases. In the first case, the prong-matching $\sigma$ is defined such that the middle component of $\Gamma$ is of abelian type. Then we obtain the condition  $e_{1}+3e_{5}=0$. In the second case, where $\sigma$ is defined so that the middle component is not of abelian type, the multi-scale $k$-differential lies in the closure of the isoresidual locus without imposing any further condition.
\end{ex}
\par
We can now state the main result of this section as follows.
\begin{thm}\label{thm:closureMSD}
 A multi-scale $k$-differential lies in the closure $\overline{\mathcal{F}}_{\Lambda}$
 of $\mathcal{F}_{\Lambda}$ if and only if it satisfies the $\mathcal{E}_{\Lambda}^{k}$-GRC at every level.
\end{thm}
The proof is very similar to that of \cite[Theorem 3.2]{CGPT1}; we therefore only sketch the main idea.
\begin{proof}
 The $\mathcal{E}_{\Lambda}^{k}$-GRC condition is stronger than the usual $k$-GRC, which ensures that multi-scale $k$-differentials can be smoothed into the interior of the stratum. Therefore, it suffices to verify only the additional linear residue conditions imposed by $\mathcal{E}_{\Lambda}^{k}$. First, we note that every horizontal node can be smoothed locally and independently. Consequently, we may assume that the  multi-scale $k$-differential under consideration has only vertical nodes.

The necessity of the condition follows exactly as in the case of abelian differentials. For sufficiency, fix a level~$L$ and suppose, by induction, that all edges whose lower endpoints lie at levels strictly above $L$ have already been smoothed. Consider the connected components $Y_i$ of $\Gamma'_{>L}$; we will smooth their edges connecting to level $L$, assuming that the $\calE_{\Lambda}^{k}$-GRC is satisfied.

Consider the poles $q'_{j}$ lying in the components $Y_{i}$. To prove the result, it suffices to construct a modification $k$-differential whose residues at the nodes agree with those at the lower endpoints, and such that the residues at the nodes satisfy the equations of $\Lambda^{\vee}$. 

If the component is not of abelian type, then by \cite[Lemma 4.4]{BCGGM3} there exists such a differential with zero residues at the poles (a trivial solution of $\Lambda^{\vee}$). If the component is of abelian type with some free poles, then we can choose a modification differential with zero residue at the other poles and with values at the free poles making the residue theorem hold. In the final case, the existence follows as in the case of abelian differentials treated in \cite[Theorem 3.2]{CGPT1}.
\end{proof}

\subsection{The closure of isoresidual fibers}\label{sub:isoclosure}

Since we are interested in isoresidual fibers, we restrict to the case $\Lambda = \mathbb{C}^{\ast} \cdot \lambda$, where $\lambda = (\lambda_1,\ldots, \lambda_p)$. In this case, we abuse notation and also write $\lambda$ instead of~$\Lambda$.
Note that when $\lambda=(0,\dots,0)$, we have $\mathcal{E}_{\lambda}^{\vee}=\mathcal{R}_{p}^{\vee}$. If $\lambda$ has at least one nonzero entry, then~$\mathcal{E}_{\lambda}^{\vee}$ is the image by the projection~\eqref{eq:projection} from $\CC^{p}$ to $\mathcal{R}_{p}^{\vee}$ of the generated by equations of the form
\begin{equation}\label{eq:relres}
 f_{i,j}(q_{1},\dots,q_{p})\coloneqq\lambda_{i}r_{q_{j}}-\lambda_{j}r_{q_{i}}=0\,.
\end{equation}

Analogous to the case of abelian differentials, we obtain the following results.
\begin{cor}\label{cor:reszero}
If a multi-scale $k$-differential lies in the closure $\overline{\mathcal{F}}_{\lambda}$, then only the poles at the lowest level can have nonzero residues. Moreover, if two poles $q_i$ and $q_j$ lie at the same level, their residues satisfy the relation $f_{i,j} = 0$ in \eqref{eq:relres}.
\end{cor}

\begin{cor}\label{cor:reszeros}
Given $\lambda\in \mathcal R_p$,
 suppose $(X,\zeta)$ lies in the isoresidual fiber $\overline{\calF}_{\lambda}$, where $X$ is smooth. Then, if one residue of $\zeta$ vanishes, all residues of $\zeta$ vanish. Conversely, the locus of such residueless differentials (up to scalar multiple) is contained in $\overline{\calF}_{\lambda}$ for every $\lambda\in \mathcal R_p$.
\end{cor}

\section{Counting via intersection theory}\label{sec:Intersection}

In this section, we use intersection theory to prove Theorems~\ref{thm:main}, \ref{thm:one-resonant}, and \ref{thm:iso-arbitrary}, 
which determine the cardinality of the fibers in the isoresidual cover of strata of $k$-differentials on the Riemann sphere with two singularities of orders not divisible by~$k$.

\subsection{Preliminaries}

Given a stratum of abelian differentials $\omoduli[0](\mu)$ of genus zero with a unique zero,
the function 
$$f(a,p)\coloneqq\frac{a!}{(a-(p-2))!}$$ 
represents the degree of the top intersection product $\int_{\text{MS}(\mu)}\xi^{p-2}$, where  $\xi=c_1(\mathcal{O}_{\text{MS}(\mu)}(1))$ is the universal line bundle class of the multi-scale compactification $\text{MS}(\mu)\coloneqq\mathbb{P}\Xi\overline{\mathcal{M}}_{0,p+1}(\mu)$ as proved in \cite{ChPrIso}. The purpose of this subsection is to extend this result to the setting of $k$-differentials and to study the properties of the counting function that generalizes~$f$.
This relies on the notion of multi-scale $k$-differentials recalled in Section~\ref{sub:multi-scale}.

\begin{prop}\label{prop:k-minimal}
Let ${\text{MS}}^k(\mu)$ denote the multi-scale compactification of the stratum of $k$-differentials of signature $\mu=(a,-b_1,\ldots,-b_p)$ on the Riemann sphere, with $a>-k$ and $p$ poles of order $-b_{i}\leq-k$.
 
Then the degree of the top intersection product of the universal line bundle class is given by 
\begin{equation}\label{eqn:k-factorial}
    \int_{\text{MS}^k(\mu)}\xi^{p-2}=f_k(a,p) \,,
\end{equation}
where $f_{k}(a,p)$ is defined in Equation~\eqref{eq:fk}.
\end{prop}

\begin{proof}
    If all the orders in $\mu$ are divisible by $k$, then every differential in the stratum is the $k$-th power of an abelian differential. We consider the multi-scale compactification $\text{MS}(\mu/k)$ of abelian differentials with signature $\mu/k=(a/k,-b_1/k,\ldots,-b_p/k)$ and the Chern class $\xi_{ab}=c_1(\mathcal{O}_{\text{Ms}(\mu/k)}(1))$. Since $\xi=k\xi_{ab}$ via pullback, we obtain 
    \begin{eqnarray*}
        \int_{\text{MS}^k(\mu)}\xi^{p-2} & = & k^{p-2}\int_{\text{MS}(\mu/k)}\xi_{ab}^{p-2}\\
        & = & k^{p-2}f_1(a/k,p)\\
        & = & f_k(a,p)\,.
    \end{eqnarray*}
    
    Next, suppose $\mu$ has at least two entries not divisible by $k$. For every marked pole $p_i$, there is a relation of divisor classes: 
    \begin{equation}\label{eqn:xi-pole}
        \xi = (b_i-k)\psi_{p_i} + \sum_{p_i\in \Gamma^{\perp}}t_{\Gamma}D_{\Gamma}\,,
    \end{equation}
    where $\psi_{p_i}$ is the psi-class associated to $p_i$, and the sum runs over all two-level graphs $\Gamma$ in which~$p_i$ lies on the bottom component,  denoted $\Gamma^\perp$. Here, the boundary divisor $D_{\Gamma}$ is the closure of all boundary points whose dual graph is $\Gamma$, and the twisting coefficient $t_{\Gamma}$ is the least common multiple of the prongs at the nodes.
    
    By expanding the left-hand side of Equation $\eqref{eqn:k-factorial}$ using the relation $\eqref{eqn:xi-pole}$, we obtain 
    \begin{equation}\label{eqn:xi-expanded}
        \int_{\text{MS}^k(\mu)}\xi^{p-2} = (b_i-k)\int_{\text{MS}^k(\mu)}\xi^{p-3}\psi_{p_i} + t_{\Gamma_{z,p_i}}\int_{\text{MS}^k(\mu)}\xi^{p-3}D_{\Gamma_{z,p_i}}\,,
    \end{equation}
    where the graph $\Gamma_{z,p_i}$ has the zero $z$ and the pole $p_i$ in the bottom component, with the remaining poles in a single top component. This is the only graph for which the product $\xi^{n-3}D_{\Gamma}$ does not vanish. Its twisting coefficient is $t_{\Gamma_{z,p_i}}=a-b_i+k$, and by induction on the number of poles, we have  $$\int_{\text{MS}^k(\mu)}\xi^{p-3}D_{\Gamma_{z,p_i}}= f_k(a-b_i,p-1)\,.$$ 
    
    We can further expand the product $\int_{\text{MS}^k(\mu)}\xi^{p-3}\psi_{p_i}$ using the relation $\eqref{eqn:xi-pole}$, as in the proof of \cite[Theorem 1.1]{ChPrIso}, to observe that $\int_{\text{MS}^k(\mu)}\xi^{p-2}$ is a polynomial of degree at most $p-2$ in the ring $\mathbb{Z}[b_1,\ldots,b_p]$, by taking $a=\sum_{i=1}^p b_i -2k$ and considering the pole orders as variables. In particular, evaluating at $b_i=k$ gives
    \begin{equation*}
        \int_{\text{MS}^k(\mu)}\xi^{p-2}\Bigg|_{b_i=k} = a'f_k(a'-k,p-1)=f_k(a',p)\,,
    \end{equation*}
    where $a'=a|_{b_i=k}$.
    This implies that $b_i-k$ divides the polynomial $\int_{\text{MS}^k(\mu)}\xi^{n-2}-f_k(a,p)$. Moreover, since any other pole can be used in the same way, the product $\prod_{i=1}^p(b_i-p)$ divides the same polynomial,  
    whose degree is at most $p-2$. Therefore, the polynomial must vanish identically.
\end{proof}
 The function~$f_{k}(a,n)$ satisfies two identities that  will be used in the proof of Theorem~\ref{thm:main}. For the first identity, let $f_k^{(r)}(a,n)$ denote the $r$-th derivative of $f_k(a,n)$ with respect to $a$, with the convention that $f_k^{(r)}(a,n)=0$ whenever $r<0$ or $r>n-2$.
\begin{prop}\label{prop:fk-split} We have 
    \begin{equation*}
    \ f_k(a_2,n+1) = \sum_{r=0}^{n-1}(-(a_1+k))^r\left[\frac{f_k^{(r-1)}(a_1+a_2,n)}{(r-1)!}+(a_1+a_2+k)\frac{f_k^{(r)}(a_1+a_2,n)}{r!}\right]\,. 
\end{equation*}
\end{prop}
\begin{proof} We have 
    \begin{eqnarray*}
     &&\sum_{r=0}^{n-1}(-(a_1+k))^r\left[\frac{f_k^{(r-1)}(a_1+a_2,n)}{(r-1)!}+(a_1+a_2+k)\frac{f_k^{(r)}(a_1+a_2,n)}{r!}\right] \\
     &=&
    a_2\sum_{r=0}^{n-2}(-(a_1+k))^r\frac{f_k^{(r)}(a_1+a_2,n)}{r!}\\
    &=&
    a_2\sum_{r=0}^{n-2}(-(a_1+k))^r\left[\frac{f_k^{(r-1)}(a_1+a_2-k,n-1)}{(r-1)!}+(a_1+a_2)\frac{f_k^{(r)}(a_1+a_2-k,n-1)}{r!}\right]\\
    &=&a_2f_{k}(a_2-k,n) \ \ \text{by induction on }n\\
    &=&f_{k}(a_2,n+1)\, .
    \end{eqnarray*}
\end{proof}

The second identity is a sign-alternating sum of polynomials that equals the zero polynomial. Recall that given a subset $I$ of $\{1,\ldots,p\}$ the number $c_{1,I} = a_1-\sum_{i\in I}b_i + k$  has been introduced in~\eqref{eq:c1I}
\begin{prop}\label{prop:zero-sum}
For any $I \subset \{1,\ldots,p\}$ we have
    \begin{equation*}
     \sum_{I\subset \{1,\ldots,p\}}(-1)^{|I|}f_k\left(c_{1,I}+(|I|-1)k,p+1\right)=0\,.
\end{equation*}
\end{prop}
\begin{proof}
    Each polynomial $(-1)^{|I|}f_k(c_{1,I}+(|I|-1)k,p+1)\in \mathbb{Z}[a_1,b_1,\ldots,b_p]$ associated with a set $I\subset\{1,\ldots,p\}$ containing an index $i$ can be canceled with the polynomial associated to the set $I\backslash\{ i\}$ by evaluating at $b_i=k$. Therefore, the product $\prod_{i=1}^{p}(b_i-k)$ divides the sum $$\sum_{I\subset \{1,\ldots,p\}}(-1)^{|I|}f_k(c_{1,I}+(|I|-1)k,p+1)\,,$$ 
    which is a polynomial of degree at most $p-1$, and hence it must vanish identically.
\end{proof}

\subsection{The generic case}
Now we prove Theorem \ref{thm:main} which computes the degree of the isoresidual cover. To do this, we replicate the method of \cite{ChPrIso}, using intersection theory on the multi-scale compactification $\text{MS}^k(\mu)$ recalled in Section~\ref{sec:multi-scale}.
\smallskip
\par
Fix a projectivized stratum $\mathbb{P}\Omega^k(\mu)$ of primitive $k$-differentials on the Riemann sphere, where $\mu=(a_1,a_2,-b_1,\ldots,-b_p)$ is a partition of $-2k$. We denote by $z_i$ the singularity of order $a_i$, for $i=1,2$, and the poles $p_1,\ldots,p_p$, respectively. The pole orders $b_i \geq k$, with each $b_{i}$ divisible by~$k$
and each $a_{i}$ relatively prime to~$k$.
Since the  $b_i$ are positive multiples of~$k$, the residual map~\eqref{eq:resmap} induces a rational map defined by the $k$-residues $\mathcal{R}=[R_1:\cdots:R_p]$ of the $k$-differentials.

\begin{proof}[Proof of Theorem~\ref{thm:main}]
 The multi-scale space $\text{MS}^k(\mu)$ of $k$-differentials is equipped with a universal line bundle $\mathcal{O}_{\text{MS}^k(\mu)}(1)$ whose first Chern class is denoted by $\xi$. To describe the locus of $k$-differentials with a fixed general $k$-residue tuple $\mathcal{R}=[R_1:\cdots:R_p]$, we define the sections
\begin{eqnarray*}
    \phi_i\colon \mathcal{O}_{\mathcal{M}^k(\mu)}(-1) & \xrightarrow{} &\mathbb{C}\\
    \zeta & \mapsto & R_p \cdot {\rm Res}_{p_i}\zeta - R_i \cdot {\rm Res}_{p_p}\zeta\,.
\end{eqnarray*}
The vanishing locus $D_i=V(\phi_i)$ is a divisor whose cycle class is $[D_i]=c_1(\mathcal{O}_{\text{MS}^k(\mu)}(1))=\xi$.  

Since the residue tuple $\mathcal{R}$ is generic, the isoresidual fiber coincides with  $D_1\cap\cdots \cap D_{p-1}$. This intersection is transversal, since the residue equations defining the sections $\phi_i$ are linearly independent, and hence it can be computed as the intersection product $\int_{\text{MS}^k(\mu)}\xi^{p-1}$. Moreover, the intersection lies entirely in the interior of  $\text{MS}^k(\mu)$ by the generality of $\mathcal{R}$. Indeed, if there existed a multi-scale $k$-differential on a nodal curve with residues given by $\mathcal{R}$, then at least one component would carry a $k$-differential $\zeta$ whose singularity orders are all divisible by $k$. This would imply the existence of a global abelian differential $\omega$ such that $\omega^k=\zeta$. Applying the residue theorem to  $\omega$ would then yield the relation $P(\{R_i\}_{i\in I})=0$, where $I$ indexes the poles on that component,  contradicting the generality of $\mathcal{R}$.

Next, to compute the product $\int_{\text{MS}^k(\mu)}\xi^{p-1}$, we use the relation
\begin{equation}\label{eqn:xi-taut}
    \xi = -(a_1+k)\psi_{z_1} + \sum_{z_1\in \Gamma^{\perp}}t_{\Gamma}\delta_{\Gamma}\,,
\end{equation}
where $\psi_{z_1}$ is the psi-class at the marked point $z_1$,  and the sum runs over all two-level graphs $\Gamma$ whose bottom component $\Gamma^\perp$ contains $z_1$. (The argument is symmetric if we choose the marking $z_2$ instead.) Each such two-level graph determines a boundary divisor $\delta_\Gamma$ with an associated twisting coefficient $t_{\Gamma}$, given by the product of the prongs at its nodes. 

Observe that multiplying a boundary divisor $\delta_{\Gamma}$ by  $\psi_{z_1}^r\xi^{p-2-r}$ vanishes unless the bottom component has dimension $r$. This motivates the definition $$
D_{r,I}=\sum_{\Gamma \in \mathcal{G}_{r,I}}t_{\Gamma}\delta_{\Gamma}\,,$$ 
where $\mathcal{G}_{r,I}$ is the collection of two-level graphs such that:
\begin{itemize}
\item the marking $z_1$ lies on a single bottom component $X_0$,
\item the poles indexed by $I$ are distributed among $r+1$ semistable top components $X_1,\ldots,X_{r+1}$, and
\item the marking $z_2$, together with the remaining poles, lies on a single top component $X_{r+2}$. 
\end{itemize}
These graphs satisfy that the top component $X_{r+2}$ is invariant and the prong at its node is $c_{1,I}$. They are represented in Figure~\ref{fig:graphes}. 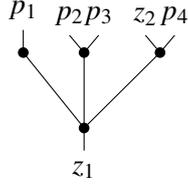
\begin{figure}[htb]
   \centering
\begin{tikzpicture}[scale=1]

\coordinate (a1) at (0,0);\fill (a1) circle (2pt);
\coordinate (a2) at (-.8,1);\fill (a2) circle (2pt);
\coordinate (a4) at (1,1);\fill (a4) circle (2pt);
\coordinate (a3) at (0,1);\fill (a3) circle (2pt);
\draw (a1) -- (a2);
\draw (a1) -- (a3);
\draw (a1) -- (a4);
\draw (a1) --++ (-90:.3) node[below] {$z_{1}$};
\draw (a2) --++ (90:.3) node[above] {$p_{1}$};
\draw (a3) --++ (50:.3)node[above] {$p_{3}$};
\draw (a3) --++ (130:.3)node[above] {$p_{2}$};
\draw (a4) --++ (50:.3)node[above] {$p_{4}$};
\draw (a4) --++ (130:.3)node[above] {$z_{2}$};

 \end{tikzpicture}
\caption{The graph in $\mathcal{G}_{r,I}$ with $I=\lbrace 1 \rbrace \cup \lbrace 2,3 \rbrace$.}
\label{fig:graphes}
\end{figure}
\par

Note that the restriction of $\delta_{\Gamma}$ to $X_{r+2}$ is the same for every $\Gamma \in \mathcal{G}_{r,I}$ and for all $r\geq 0$. By semistability of the top components, each marked point in the bottom component other than $z_1$ can be replaced by a top component containing the marked point. Consequently, the bottom component only contains $z_1$ and $r+2$ nodes. This convention behaves well with our counts as the prong of a semistable component containing the pole $p_i$ would be $b_i-k$ which multiplied by the function $f_k(b_i-2k,1)=1/(b_i-k)$ cancel each other. For the divisor $D_{r,I}$ to be well defined, the $k$-differential on $X_{r+2}$ should have a singularity of order $>-k$ at the node joining $X_{r+2}$ to $X_0$. This condition is equivalent to
\begin{equation*}
    c_{1,I}\coloneqq a_1-\sum_{i\in I}b_i+k>0\,,
\end{equation*}
which defines a wall-crossing inequality. For notational convenience, we will also allow $I=\{1,\ldots,p\}$, which corresponds to the special case where $z_2$ is the only marked point on $X_{r+2}$ (equivalently, $z_2$ lies on the bottom component). Although this case does not satisfy $c_{1,I}>0$, we will justify below why the subsequent computations still apply. Applying relation $\eqref{eqn:xi-taut}$ iteratively, we obtain 
\begin{eqnarray*}
    \int_{\text{MS}^k(\mu)}\xi^{p-1} & = & \sum_{\substack{c_{1,I'}>0 \\ I'\not=\emptyset}}\sum_{r=0}^{|I'|-1}(-(a_1+k))^r\int_{\text{MS}^k(\mu)}\psi_{z_1}^r\xi^{p-2-r}D_{r,I'}\,.
\end{eqnarray*}
To compute each product $$\int_{\text{MS}^k(\mu)}\psi_{z_1}^r\xi^{p-2-r}D_{r,I'}\,,$$ for every two-level graph $\Gamma\in \mathcal{G}_{r,I'}$, 
let $\alpha_i$ denote the zero order on $X_i$ at the node $q_i=X_i\cap X_0$. For each top component $X_i$, we associate  a restricted stratum $\text{MS}^k(\mu_i)$ of $k$-differentials, where $\mu_i=(\alpha_i,\{-b_{j}\}_{p_j\in X_i})$. On the bottom component, the restriction is  $\delta_{\Gamma}|_{X_0}\simeq \overline{\mathcal{M}}_{0,r+3}$. Hence we may decompose $\delta_{\Gamma}\simeq \mathbb{P}\mathcal{E}\times \overline{\mathcal{M}}_{0,r+3}$, where $\mathcal{E}=\eta_1\oplus\cdots\oplus\eta_{r+2}$ and $\eta_i$ denotes the tautological line bundle class over  $\text{MS}^k(\mu_i)$. Recall that the Segre class is given by 
\begin{equation*}
    s(\mathcal{E}) =\prod_{i=1}^{r+2}(1+\xi_i+\xi_i^2+\cdots+\xi_i^{d_i})\,,
\end{equation*}
where 
$d_i=\dim \text{MS}^k_i(\mu_i)$ for $i=1,\ldots,r+2$. Let $n_i$ denote the number of marked poles on $X_i$. Then $d_i=n_i-2$ for $i=1,\ldots,r+1$, while $d_{r+2}=n_{r+2}-1=p-1-|I'|$. Since
    $\sum_{i=1}^{r+2} d_i = p+1-2(r+2)$, and $\text{rank} \ \mathbb{P}\mathcal{E}=r+1$ over $\text{MS}^k(\mu_1)\times\cdots \times \text{MS}^k(\mu_{r+2})$, we obtain 
\begin{eqnarray*}
    \int_{\mathbb{P}\mathcal{E}} \xi^{p-2-r}&=&s_{p-3-2r}(\mathcal{E})\\
    &=&\prod_{i=1}^{r+2} \int_{\text{MS}^k(\mu_i)}\xi_i^{d_i}\\    &=&\left[\int_{\text{MS}^k(\mu_{r+2})}\xi_{r+2}^{d_{r+2}}\right]\prod_{i=1}^{r+1}f_k(\alpha_i,n_i) \ \  \text{by Proposition \ref{prop:k-minimal}}\,.
\end{eqnarray*}
By the induction hypothesis on $\text{MS}^k(\mu_{r+2})$, since  $\alpha_{r+2}=c_{1,I'}-k$ and $n_{r+2}=p-|I'|$, we have 
\begin{equation*}  \int_{\text{MS}^k(\mu_{r+2})}\xi_{r+2}^{d_{r+2}} = \sum_{c_{1,I'\sqcup I''}>0} c_{1,I'\sqcup I''}f_k(c_{1,I'}-k,|I''|+1)f_k(a_2,p+1-|I'\sqcup I''|)\,.
\end{equation*}
This expression is the same for every $\Gamma \in \mathcal{G}_{r,I'}$ and for all $r\geq 0$. Recall that each $\Gamma \in \mathcal{G}_{r,I'}$ carries a twisting coefficient $t_{\Gamma}=\prod_{i=1}^{r+2} t_i$, where $t_i=\alpha_i+k$ is the prong at the node $q_i$. In particular, $t_{r+2}=c_{1,I'}$, and we set $t'_{\Gamma}=\prod_{i=1}^{r+1} t_i$. Thus $t_{\Gamma}=c_{1,I'}t'_{\Gamma}$. Putting everything together, we obtain
\begin{eqnarray*}    \int_{\text{MS}^k(\mu)}\psi_{z_1}^r\xi^{p-2-r}D_{r,I'} & = & \sum_{\Gamma \in \mathcal{G}_{r,I'}}c_{1,I'}t'_{\Gamma}\int_{\text{MS}^k(\mu)}\psi_{z_1}^r\xi^{p-2-r}\delta_{\Gamma}\\
&=& \sum_{\Gamma \in \mathcal{G}_{r,I'}}c_{1,I'}t'_{\Gamma}\int_{\overline{\mathcal{M}}_{0,r+3}}\psi_{z_1}^r\int_{\mathbb{P}(\mathcal{E})}\xi^{p-2-r}\\
&=& \left[\int_{\text{MS}^k(\mu_{r+2})}\xi_{r+2}^{d_{r+2}}\right]\sum_{\Gamma \in \mathcal{G}_{r,I'}}t'_{\Gamma}\prod_{i=1}^{r+1}f_k(\alpha_i,n_i)\,.
\end{eqnarray*}
Now, consider the minimal stratum $\text{MS}^k(\mu')$ with $\mu'=(B_{I'}-2k,\{-b_i\}_{i\in I'})$, where $B_{I'}=\sum\limits_{i\in I'}b_i$. Let~$z$ be the single zero and $\xi=c_1(\mathcal{O}_{\text{MS}^k(\mu')}(1))$. Using the relation \eqref{eqn:xi-taut} with respect to the zero $z$, we obtain
\begin{equation*}    \int_{\text{MS}^k(\mu')}\psi_z^{r}\xi^{|I'|-2-r} = -(B_{I'}-k)\int_{\text{MS}^k(\mu')}\psi_z^{r+1}\xi^{|I'|-3-r} + \sum_{\Gamma \in \mathcal{G}^{min}_{r,I'}}t_{\Gamma}\int_{\text{MS}^k(\mu')}\psi_z^{r}\xi^{|I'|-3-r}\delta_{\Gamma}\,,
\end{equation*}
where $\mathcal{G}^{min}_{r,I'}$ runs over all two-level graphs whose bottom component contains a single zero $z$ of order $a_{I'}=a_1+a_2-B_{I'}$, and the top level consists of $r+1$ semistable top components. 

Note that we can ``minimize'' a graph $\Gamma\in \mathcal{G}_{r,I'}$ to $\Gamma^{\min} \in \mathcal{G}^{min}_{r,I'}$ by collapsing the top component $X_{r+2}$ and the zero $z_1$ into a single zero $z$ in the bottom component of $\Gamma^{min}$. This process is bijective. Let $\text{MS}(\mu'/k)$ be the stratum of abelian differentials with $\mu'/k=(B_{I'}/k-2,\{-b_i/k\}_{i\in I'})$. Then we can compute
\begin{eqnarray*}
    \sum_{\Gamma \in \mathcal{G}_{r,I'}}t'_{\Gamma}\prod_{i=1}^{r+1}f_k(\alpha_i,n_i) & = & \sum_{\Gamma \in \mathcal{G}^{min}_{r,I'}}t_{\Gamma}\int_{\text{MS}^k(\mu')}\psi_z^{r-1}\xi^{|I'|-2-r}\delta_{\Gamma}\\
    &=& \int_{\text{MS}^k(\mu')}\psi_z^{r-1}\xi^{|I'|-1-r} + (B_{I'}-k)\int_{\text{MS}^k(\mu')}\psi_z^{r}\xi^{|I'|-2-r}\\
    &=& k^{|I'|-2}\left[\int_{\mathcal{M}(\mu'/k)}\psi_z^{r-1}\xi^{|I'|-1-r} + (B_{I'}/k-1)\int_{\mathcal{M}(\mu'/k)}\psi_z^{r}\xi^{|I'|-2-r}\right]\\
    &=& k^{|I'|-2}\left[\frac{f^{(r-1)}(B_{I'}/k-2,|I'|)}{(r-1)!}+(B_{I'}/k-1)\frac{f^{(r)}(B_{I'}/k-2,|I'|)}{r!}\right] \text{ by \cite[Theorem 6.1]{ChPrIso}}\\
    &=& \frac{f_k^{(r-1)}(B_{I'}-2k,|I'|)}{(r-1)!}+(B_{I'}-k)\frac{f_k^{(r)}(B_{I'}-2k,|I'|)}{r!}\,.
\end{eqnarray*}
Recall that $f_k^{(r)}(a,n)$ denotes the $r$-th derivative of $f_k(a,n)$ with respect to $a$. We then obtain 
\begin{eqnarray*}
    \int_{\text{MS}^k(\mu)}\xi^{p-1} 
    & = & \sum_{\substack{c_{1,I'}>0 \\ I'\not=\emptyset}}c_{1,I'}\left[ \sum_{c_{1,I'\sqcup I''}>0} c_{1,I'\sqcup I''}f_k(c_{1,I'}-k,|I''|+1)f_k(a_2,p+1-|I'\sqcup I''|)\right]\\
    &  & \ \  \ \ \ \ \ \ \ \ \ \ \ \ \  \cdot \sum_{r=0}^{|I'|-1}(-(a_1+k))^r\left[\frac{f_k^{(r-1)}(B_{I'}-2k,|I'|)}{(r-1)!}+(B_{I'}-k)\frac{f_k^{(r)}(B_{I'}-2k,|I'|)}{r!}\right]\hfill\\
    & = & \sum_{\substack{c_{1,I}>0 \\ I=I'\sqcup I''}} c_{1,I}f(a_2,p+1-|I|) \left[\sum_{\substack{I'\subset I \\ I'\not= \emptyset}}f_k(c_{1,I'},|I|-|I'|+2)f_k(B_{I'}-a_1-2k,|I'|+1)\right] \hfill \text{by \eqref{prop:fk-split}}\\
    & = & \sum_{c_{1,I}>0} c_{1,I}f_k(a_2,p+1-|I|) \left[\sum_{\substack{I'\subset I \\ I'\not= \emptyset}}(-1)^{|I'|-1}f_k(c_{1,I'},|I|-|I'|+2)f_k(c_{1,|I'|}+(|I'|-1)k,|I'|+1)\right]\\
    & = & \sum_{c_{1,I}>0} c_{1,I}f_k(a_2,p+1-|I|) \left[\sum_{\substack{I'\subset I \\ I'\not= \emptyset}}(-1)^{|I'|-1}f_k(c_{1,I'}+(|I'|-1)k,|I|+1)\right]\\
    & = & \sum_{c_{1,I}>0} c_{1,I}f_k(a_2,p+1-|I|) \left[f_k(a_1,|I|+1)-\sum_{I'\subset I}(-1)^{|I'|-1}f_k(c_{1,I'}+(|I'|-1)k,|I|+1)\right] \hfill \text{by (\ref{prop:zero-sum}})\\
    & = & \sum_{c_{1,I}>0} c_{1,I}f_k(a_1,|I|+1)f_k(a_2,p+1-|I|)\,.
\end{eqnarray*}
\end{proof}
Note that as mentioned at the beginning of the proof, the same procedure can be repeated using the relation \eqref{eqn:xi-taut} with respect to the zero $z_2$. In that case one obtain
$$\sum_{c_{2,I}>0} c_{2,I}f_k(a_2,|I|+1)f_k(a_1,p+1-|I|)\,.$$

We can now prove the special case in which all pole orders are equal to $k$, as stated in Corollary~\ref{cor:polesk}. 

\begin{proof}[Proof of Corollary~\ref{cor:polesk}]
Let $d(a_{1},a_{2})$ denote the degree of the isoresidual cover of $\Omega^{k}\mathcal{M}_{0}(a_{1},a_{2},\rec[-k][p])$. In the formula proved in Theorem~\ref{thm:main}, for any subset $I\subset \{1,\ldots,p\}$ we have $c_{1,I}=a_{1}-|I|\cdot k + k$. In particular, $c_{1,I}>0$ if and only if $|I| \leq \ell_{1}$, where $\ell_{i}=\lceil a_{i}/k \rceil$. Summing together all terms corresponding to subsets of the same cardinality, we obtain
$$
d(a_{1},a_{2})
=
\sum\limits_{i=0}^{\ell_{1}}
\binom{p}{i}
\cdot
(a_{1}-(i-1)k)
\cdot
f_{k}(a_{1},i+1)
\cdot
f_{k}(a_{2},p-i+1)\,.
$$
Observe that 
$$(a_{1}-(i-1)k)
\cdot
f_{k}(a_{1},i+1)=f_{k}(a_{1},i+2)\,,$$ 
so the previous formula simplifies to 
$$
d(a_{1},a_{2})
=
\sum\limits_{i=0}^{\ell_{1}}
\binom{p}{i}
\cdot
f_{k}(a_{1},i+2)
\cdot
f_{k}(a_{2},p-i+1)\,.
$$
Since $f_{k}(a_{2},p+1) = a_{2}!_{(k)}$, this gives
\begin{eqnarray*}
 d(a_{1},a_{2}) &=& \sum\limits_{i=0}^{\ell_{1}}
\binom{p}{i}
\cdot
\frac{a_{1}!_{(k)}}{a_{1}\dots (a_{1}-(\ell_{1}+i-1)k)}
\cdot
a_{2}!_{(k)} \prod_{j=0}^{i} a_{2}-(\ell_{2}+j)k \\
&=&  a_{1}!_{(k)} a_{2}!_{(k)}\sum\limits_{i=0}^{\ell_{1}}
\binom{p}{i}
\cdot
\frac{\prod_{j=0}^{i} a_{2}-(\ell_{2}+j)k}{a_{1}\dots (a_{1}-(\ell_{1}+i-1)k)}  \\
&=& a_{1}!_{(k)}a_{2}!_{(k)} \sum\limits_{i=0}^{\ell_{1}}\left((-1)^{p-i}\binom{p}{i} \right)\\
&=& a_{1}!_{(k)}a_{2}!_{(k)} \binom{p-1}{\ell_{1}}\,,
\end{eqnarray*}
where the passage from the second line to the third uses the relation $$a_{1}-(\ell_{1}-j-1)k = -(a_{2}-(\ell_{2}+j)k)\,,$$ 
and the final equality follows from a telescoping sum.
\end{proof}

\subsection{The resonant case}
Now we consider the case of $k$-residues lying in a  resonance hyperplane. This hyperplane is defined by a $resonant$ subset $I\subset \{1,\ldots,p\}$, indexing $k$-residues such that $P(\{R_i\}_{i\in I_1})=0$. Recall that the notation~\eqref{eq:c1I} is $c_{i,I}=a_i + k -\sum\limits_{j\in I}b_j$. The following proposition recalls Theorem~\ref{thm:one-resonant}.

\begin{prop}\label{prop:one-psv}
    The cardinality of the isoresidual fiber over a $k$-residue tuple $\mathcal{R}=[R_1:\cdots:R_p]$ that satisfies exactly one resonance hyperplane defined by a subset $I\subseteq\{1,\ldots,p\}$ is a piecewise polynomial of degree $p-1$, given by  
    \begin{eqnarray*}
        d_{k}(\mu)- f_{I}\cdot \Ab_{\mathcal{R}}(I)\,,&  \ \ if \ \ I = \{1,\ldots,p\}\,,\\
        d_{k}(\mu)- \left(c_{1,I_{>0}}d_k(\mu_{I,\emptyset})+c_{2,I_{>0}}d_k(\mu_{\emptyset, I})\right)f_I \cdot \Ab_{\mathcal{R}}(I)\,, & \ \ if \ \ I\subsetneq \{1,\ldots,p\}\,, 
    \end{eqnarray*}
where 
\begin{itemize}
\item $f_I=f(B_I/k-1,|I|+1)$ with $B_I= \sum_{i\in I}b_{i}$, 
\item the reduced orders are $\mu_{\emptyset,I}=(a_1,c_{2,I}-k,\{-b_i\}_{i\in I^c})$ and $\mu_{I,\emptyset}=(c_{1,I}-k,a_2,\{-b_i\}_{i\in I^c})$, 
\item the notation $c_{i,I_{>0}}$ indicates that the coefficient is included only if $c_{i,I}>0$, and 
\item the function $\Ab_{\mathcal{R}}(I)$ is defined as 
\begin{equation*}
       \Ab_{\mathcal{R}}(I)=\#\left\{r_I=\{r_i\}_{i\in I} \ \Bigg| \ \sum_{i\in I}r_i=0, \  r_i^k=R_i  \ \forall i\in I\right\}\Bigg/\mathbb{C}^*\,.
   \end{equation*}
   \end{itemize}
\end{prop}

Before giving the proof of this result we illustrate the fact that $\Ab_{\mathcal{R}}(I)$ may be $>1$.
\begin{ex}\label{ex:abnumber}
In the case of $k=3$, $\mu=(4,-1;-3,-3,-3)$, and $\mathcal{R}=[1:1:1]$, the residue tuple $\mathcal{R}$ lies over the resonance subset $I=\{1,2,3\}$, but the corresponding abelian residue tuples are $r_{I,1}=[1:\zeta:\zeta^2]$ and $r_{I,2}=[1:\zeta^2:\zeta]$, where $\zeta^3=1$ is a non-trivial cubic root of unity. Hence  in that case $\Ab_{\mathcal{R}}(I)=2$.
\end{ex}

\begin{proof}
Assume, without loss of generality, that $R_p\not=0$. We can construct a chain of spaces $\mathcal{M}_0 \supset \mathcal{M}_1\supset \cdots \supset \mathcal{M}_{p-1}$, where $\mathcal{M}_0=\text{MS}^k(\mu)$ is the full multi-scale space, and    for $i=1,\ldots,p-1$ the space $\mathcal{M}_i$ is the closure of the interior points of $\mathcal{M}_{i-1}$ vanishing on the section
    \begin{eqnarray*}
        f_i\colon\mathcal{O}_{\mathcal{M}_0}(-1) & \rightarrow & \mathbb{C} \\
        \zeta & \mapsto & R_p {\rm Res}_{p_i}\zeta - R_i {\rm Res}_{p_p}\zeta\,.
    \end{eqnarray*}
 In particular, $\mathcal{M}_i$ fixes the first $i$ residues with respect to $R_p$ up to scaling. The class $[\mathcal{M}_{i}]$ is then obtained by intersecting the class of the previous space $[\mathcal{M}_{i-1}]$ with the divisor class
    \begin{equation*}
        D_i = \xi - \sum_{\Gamma \in \mathcal{G}_{i}}t_{\Gamma}D_{\Gamma}\,,
    \end{equation*}
    where $\xi=c_1(\mathcal{O}_{\text{MS}(\mu)}(1))$, and $\mathcal{G}_i$ parameterizes all two-level graphs where the $i$-th residue is fixed by the previous $i-1$ residues together with the $k$-GRC of the graph. 
    
    Assume, without loss of generality, that the pole index $p-1 \in I$. Since there is only one resonance subset $I$, we have $D_i=\xi$ for $i=1,\ldots,p-2$. Therefore, the cardinality of the concerned isoresidual fiber is given by the degree
   \begin{equation*}
       \deg \mathcal{M}_{p-1} = \int_{\mathcal{M}_{p-2} }\xi-\sum_{\Gamma \in \mathcal{G}_{p-1}}t_{\Gamma}\int_{\mathcal{M}_{p-2} }\delta_{\Gamma}\,.
   \end{equation*}
  On the right-hand side, we know that
   \begin{eqnarray*}
       \int_{\mathcal{M}_{p-2} }\xi & = & \int_{\text{MS}^k(\mu)}\xi^{p-1} = d_k(\mu)\\
       \int_{\mathcal{M}_{p-2} }\delta_{\Gamma} & = & \int_{\text{MS}^k(\mu)}\xi^{p-2}\delta_{\Gamma} \,.
   \end{eqnarray*}
   Moreover, the only two-level graphs in $\mathcal{G}_{p-1}$ whose boundary divisor $\delta_{\Gamma}$ does not vanish when multiplied by $\xi^{p-2}$ are $\Gamma_{1,I}$ and $\Gamma_{2,I}$, where $\Gamma_{i,I}$ has the marking $z_i$ in the bottom component, the poles indexed by $I$ in a top component, and the remaining markings in a second top component. Note that such a graph only appears if $c_{i,I}>0$. In the case $I=\{1,\ldots,p\}$, the only non-vanishing two-level graph  is $\Gamma_{z_1,z_2}$, which contains  the markings $z_1,z_2$ in the bottom component and all $p$ poles in a single top component. 
   
   For any two-level $\Gamma$ of these types, the boundary divisor $\delta_{\Gamma}$ restricted to the top component containing the poles indexed by $I$ is isomorphic to $\text{MS}^k(\mu_I)$, with $\mu_I=(B_I-2k,\{-b_i\}_{i\in I})$. Since all these orders are divisible by $k$, every $k$-differential in this top component is the $k$-th power of an abelian differential with residues $r_I=\{r_i\}_{i\in I}$ satisfying $\sum_{i\in I}r_i=0$ and $r_i^k=R_i$ for all $i\in I$. The prong number at the node of this top component is then $B_I/k-1$.
   
   By \cite{GeTaIso} and \cite{ChPrIso}, the cardinality of the isoresidual fiber in the abelian multi-scale space $\text{MS}(\mu_I/k)$ over a generic residue tuple $r_I$ is given by
   \begin{equation*}            \int_{\text{MS}(\mu_I/k)}\xi_{ab}^{|I|-2} = f(B_I/k-2,|I|)\,,
   \end{equation*}
   where $\xi_{ab}=c_1(\mathcal{O}_{\text{MS}(\mu_I/k)}(1))$ and $\mu_I/k$ denotes $\mu_I$ with all orders divided by $k$. However,  since the corresponding residue tuple $r_I$ may not be unique  up to scaling (see Example~\ref{ex:abnumber}) we obtain
   \begin{equation*}            \int_{\text{MS}^k(\mu_I)}\xi^{|I|-2} = \Ab_{\mathcal{R}}(I)\int_{\text{MS}(\mu_I/k)}\xi_{ab}^{|I|-2}\,.
   \end{equation*}
   For $I\subsetneq \{1,\ldots,p\}$, we have 
   \begin{eqnarray*}
   t_{\Gamma_{1,I}}\int_{\text{MS}^k(\mu)}\xi^{p-2}\delta_{\Gamma_{1,I}} & = &  (B_{I}/k-1)c_{1,I}f(B_{I}/k-2,|I|)d_k(\mu_{I,\emptyset})\Ab_{\mathcal{R}}(I)\,,\\
    t_{\Gamma_{2,I}}\int_{\text{MS}^k(\mu)}\xi^{p-2}\delta_{\Gamma_{2,I}} & = &  (B_{I}/k-1)c_{2,I}f(B_{I}/k-2,|I|)d_k(\mu_{\emptyset,I})\Ab_{\mathcal{R}}(I)\,.
 \end{eqnarray*}
For $I=\{1,\ldots,p\}$, we have 
   \begin{equation*}
     t_{\Gamma_{z_1,z_2}}\int_{\text{MS}^k(\mu)}\xi^{p-2}\delta_{\Gamma_{z_1,z_2}}  =  ((a_1+a_2)/k+1)f((a_1+a_2)/k,p)\Ab_{\mathcal{R}}(I)\,.
   \end{equation*}
   Combining the above, the desired conclusion follows.
\end{proof}

Finally, we prove the most general case stated in Theorem~\ref{thm:iso-arbitrary} recalled in the following proposition.
\begin{prop}
    The cardinality of the isoresidual fiber over an arbitrary tuple of residues $\mathcal{R}=[R_1:\cdots:R_p]$ is
    \begin{equation*}
        \sum_{J_0\sqcup J_1\sqcup \cdots \sqcup J_s}(-1)^sG_k(J_0;J_1,\ldots,J_s)\prod_{j=1}^{s_i}f_{J_j}\Ab_{\mathcal{R}}(J_j)\,,
    \end{equation*}
    where the sum runs over all partitions $J_0 \sqcup \cdots \sqcup J_s=\{1,\ldots,p\}$ such that $J_i$ is resonant for all $i=1,\ldots,s$, the remaining subset $J_0$ is possibly empty, and $J'_1,\ldots,J'_{s'}$ is an irreducible resonant subpartition of $J_1,\ldots,J_s$. 
    The functions appearing above are defined as 
    \begin{eqnarray*}
        G_k(J_0;J_1,\ldots,J_s) & = & \begin{cases} 
      \sum_{d_{1,I}>0 } d_{1,I}(a_1+k)^{|I|-1}(a_2+k)^{|I^c|-1}\,, & \text{if} \ J_0=\emptyset\,, \\
      
      \sum_{\substack{d_{1,I}>0 \\ d_{2,I^c}>0}} d_{1,I}d_{2,I^c}d_k(\mu_{I,I^c})(a_1+k)^{|I|-1}(a_2+k)^{|I^c|-1}\,, & \text{if} \ J_0\not=\emptyset\,, 
   \end{cases} \\
        \mu_{I,I^c} & = & (d_{1,I}-k,d_{2,I^c}-k,\{-b_i\}_{i\in J_0})\,,\\
        d_{i,I} & = & c_{i, \ \sqcup_{j\in I} J_{j}}, \ \ \text{for} \ \   i=1,2 \ \  \text{and} \ \  I\subseteq\{1,\ldots,s\}\,.
    \end{eqnarray*}
\end{prop}

\begin{proof}
    We define the spaces $\mathcal{M}_0 \supset \mathcal{M}_1 \supset \cdots\supset\mathcal{M}_{p-1}$ as in the preceding proof. Then
    \begin{equation*}
       \deg \mathcal{M}_{p-1} = \int_{\mathcal{M}_{p-2}}\xi-\sum_{\Gamma \in \mathcal{G}_{p-1}}t_{\Gamma}\int_{\mathcal{M}_{p-2}}\delta_{\Gamma}\,,
   \end{equation*}
   where we interpret the right-hand side as a combination of intersection numbers. 
   The first term on the right-hand side corresponds to the isoresidual fiber over a residue tuple $\mathcal{R}'$ in which all residues are equal to those in $\mathcal{R}$, except for the $(p-1)$-th residue, which does not belong to any resonance subset. By induction on the number of resonance subsets, we obtain 
   \begin{equation*}
       \int_{\mathcal{M}_{p-2}}\xi = \sum_{\substack{J_0\sqcup J_1\sqcup \cdots \sqcup J_s \\ p-1 \in J_0}}(-1)^sG_k(J_0;J_1,\ldots,J_s)\prod_{j=1}^{s_i}f_{J_j}\Ab_{\mathcal{R}}(J_j)\,.
   \end{equation*}
   
    In the case where $R_{p-1}$ does not belong to any resonance subset, we may simply re-index the poles and choose a residue that does. On the other hand, to compute  the degree of $\int_{\mathcal{M}_{p-2}}\delta_{\Gamma}$, assume that the two-level graph $\Gamma$ has $m+1$ top components $X_0;X_1,\ldots,X_m$ and a bottom component $X_b$, where $X_0$ and $X_b$ contain the marked points $z_i$ and $z_j$, respectively, for $i,j=1,2$. We allow $X_0$ to be empty, which corresponds to the situation where both zeroes $z_1$ and $z_2$ lie on the bottom component $X_b$.  
    
    Fixing the residues $R_1,\ldots,R_{p-2},R_p$ imposes resonant conditions on every irreducible component as follows. If a resonant subset is properly contained in a top component, it imposes a condition on the differential restricted to that component. If a resonant subset indexes all the poles in that component and does not  contain $p-1$, then the residue at the corresponding node vanishes for the differential on $X_b$. 
    
    Let $\mathcal{M}^{\Gamma}_i$ denote the restriction of the divisor $\delta_{\Gamma}|_{X_i}$ with the residue conditions imposed by $\mathcal{M}_{p-2}$. Denote by $d_b$ the dimension of the space $\mathcal{M}^{\Gamma}_b$ corresponding to the bottom component. Then  
    \begin{equation*}
        \int_{\mathcal{M}_{p-2}}\delta_{\Gamma} = 
        \begin{cases} 
      \deg \mathcal{M}^{\Gamma}_b \prod_{i=0}^{m}\deg \mathcal{M}^{\Gamma}_i\,, & \text{if} \ d_b=0\,, \\      
      0\,, & \text{if} \ d_b>0\,.
   \end{cases}       
    \end{equation*}

Observe that each space $\mathcal{M}^{\Gamma}_i$ for $i=1,\ldots,m$ corresponds to a stratum $\komoduli[0](\mu_{i})$ with orders $\mu_i=(B_{X_i}-2k,\{-b_j\}_{p_j\in X_i})$, where $B_{X_i}=\sum_{p_j\in X_i}b_j$, and residues $\mathcal{R}_i=\{R_j\}_{p_j\in X_i}$. Since all orders of $\mu_i$ are multiples of $k$, every $k$-differential $\zeta \in \mathcal{M}^{\Gamma}_i$ can be written as $\zeta=\omega^k$ for some abelian differential $\omega$ with residues $r_i=\{r_j\}_{p_j\in X_i}$ satisfying $r_j^k=R_j$ for each $j$.

For each residue tuple $r_i$, Theorem 1.2 of \cite{ChPrIso} gives that the number of isoresidual differentials is
\begin{equation*}
     \sum_{J_1\sqcup \cdots \sqcup J_{s_i}} (-1)^{s_i-1}(B_{X_i}/k-1)^{s_i-2}\prod_{j=1}^{s_i}f_{J_j}\,,
\end{equation*}
where $f_I=f(B_I/k-1,|I|+1)$, and the sum runs over all partitions of $r_i$ into $s_i$ resonant subsets $J_1,\ldots,J_{s_i}$.

As in the single-resonance case, this implies that we need to consider
\begin{equation*}
    \Ab_{\mathcal{R}}(J_1,\ldots,J_s) = \#\left\{r=\{r_i\}_{i\in \sqcup_{j=1}^sJ_j} \ \Bigg| \ \sum_{i\in J_j}r_i=0 \text{ for } j=1,\ldots,s, \  r_i^k=R_i  \ \forall i\in \sqcup_{j=1}^sJ_j\right\}\Bigg/\mathbb{C}^*\,.
\end{equation*}
This function counts the number of times the term $(-1)^{s_i-1}(B_{X_i}/k-1)^{s_i-2}\prod_{j=1}^{s_i}f_{J_j}$ appears. Since every tuple $r \in \Ab_{\mathcal{R}}(J_1,\ldots,J_s)$ can be expressed as $r=[r_1,\zeta^{i_2}r_2,\ldots,\zeta^{i_s}r_s]$,  where $r_j\in \Ab_{\mathcal{R}}(J_j)$ and $\zeta$ is a $k$-root of unity, we obtain
\begin{equation*}
    \Ab_{\mathcal{R}}(J_1,\ldots,J_s) = k^{s-1}\prod_{j=1}^s \Ab_{\mathcal{R}}(J_j) \,.
\end{equation*}
Finally, if we also include the contribution of the prong $(B_{X_i}/k-1)$ at the node $q_i=X_i\cap X_b$, the total contribution of the term becomes
\begin{equation*}
     (-1)^{s_i-1}(B_{X_i}/k-1)^{s_i-1}\left(\prod_{j=1}^{s_i}f_{J_j}\right)\Ab_{\mathcal{R}}(J_1,\ldots,J_{s_i}) = (-1)^{s_i-1}(B_{X_i}-k)^{s_i-1}\prod_{j=1}^{s_i}f_{J_j}Ab_{\mathcal{R}}(J_j) \,.
\end{equation*}

We conclude that
\begin{equation*}
    (B_{X_i}/k-1)\deg \mathcal{M}^{\Gamma}_i = (-1)^{s_i-1}(B_{X_i}-k)^{s_i-1}\prod_{j=1}^{s_i}f_{J_j}\Ab_{\mathcal{R}}(J_j) \,.
\end{equation*}
To compute the degree of the bottom component moduli $\mathcal{M}^{\Gamma}_b$, assume that $p_{p-1}\in X_m$. Then the residue at every node $q_i$ for $i=1,\ldots,m-1$ vanishes. The residueless condition at a pole $q_i$ corresponds to the divisor class $(B_{X_i}-k)\psi_{q_i}$, where $\psi_{q_i}$ is the psi-class at the node $q_i$. By the well-known formula for the product of $\psi$-classes on the moduli space of pointed rational curves, we obtain 
\begin{equation*}
    \deg \mathcal{M}^{\Gamma}_b = \prod_{i=1}^{m-1}(B_{X_i}-k)\psi_{q_i} = (m-1)!\prod_{i=1}^{m-1}(B_{X_i}-k)\,.
\end{equation*}
We can then express the intersection product as
\begin{equation*}
    \int_{\mathcal{M}_{p-2}}\delta_{\Gamma} = \sum_{\substack{J_1\sqcup \cdots \sqcup J_s \\ p-1 \in J_i \text{ for some }i}} \text{Coeff}^{\Gamma}(J_1,\ldots,J_s)\prod_{j=1}^{s_i}f_{J_j}\Ab_{\mathcal{R}}(J_j) \,.
\end{equation*}

    To compute the coefficient $\text{Coeff}^{\Gamma}(J_1,\ldots,J_s)$, we first need a combinatorial description of the graphs where the therm $\prod_{i=1}^sf_{J_i}$ appears. These are the graphs in which each subset $J_j$ is entirely contained in some top component $X_i$. This defines a $(m+1)$-partition of the indices $\{1,\ldots,s\}$, where the set $J_0\coloneqq\left(\sqcup_{i=1}^rJ_{i}\right)^c$ is completely contained in $X_0$. Note that if $X_0$ is empty, then $J_0$ is empty; however, the converse is not necessarily true. We call a graph of this type  \textit{compatible} with $J_1,\ldots,J_s$. 
    
    Assume that $p-1 \in J_s$ and that $J_s$ is contained in $X_m$. Let $K_0,\ldots,K_m$ be a partition of the indices $\{1,\ldots,s-1\}$ such that $j\in K_i$ if $J_j\subset X_i$, and $K_m$ is possibly empty (this occurs when $J_s$ is the only resonant subset in $X_m$). Then: 
    \begin{eqnarray*}
        \deg \mathcal{M}^{\Gamma}_b & = & (m-1)!\prod_{i=1}^{m-1}(B_{X_i}-k)\,, \\
        t_{\Gamma} & = & d_{i,K_0^c\cup\{s\}}\prod_{i=1}^m(B_{X_i}/k-1)\,,\\
        (B_{X_i}/k-1)\text{Coeff}^{X_i}(\{J_j\}_{j\in K_i}) & = & (-1)^{|K_i|-1}(B_{X_i}-k)^{|K_i|-1}\,, \ \ i=1,\ldots,m-1\,,\\
        (B_{X_m}/k-1)\text{Coeff}^{X_m}(\{J_j\}_{j\in K_m\cup\{s\}}) & = & (-1)^{|K_m|}(B_{X_m}-k)^{|K_m|}\,,       
    \end{eqnarray*}
    where $d_{i,I}=a_i+k-\sum_{j\in I}B_{J_j}$. It follows that
    \begin{eqnarray*}
        t_{\Gamma}\text{Coeff}^{\Gamma}(J_1,\ldots,J_s) & = & t_{\Gamma}\deg \mathcal{M}^{\Gamma}_b\left(\prod_{i=0}^{m-1} \text{Coeff}^{X_i}(\{J_j\}_{j\in K_i})\right)\text{Coeff}^{X_m}(\{J_j\}_{j\in K_m\cup\{s\}})\\
        & = & (-1)^{s-m}(m-1)!d_{i,K_0^c\cup\{s\}}G_k(J_0;\{J_j\}_{j\in K_0})\left(\prod_{i=1}^{m}(B_{X_i}-k)^{|K_i|}\right)\,.
    \end{eqnarray*}
    
    Now, fixing $K_0$ (i.e., fixing the poles in the top component $X_0$), we sum over all compatible graphs: 
    \begin{eqnarray*}
        \sum_{\substack{\Gamma \text{ compatible} \\ K_0 \text{ fixed}}} t_{\Gamma}\text{Coeff}^{\Gamma}(J_1,\ldots,J_s) & = & d_{i,K_0^c\cup\{s\}}G_k(J_0;\{J_j\}_{j\in K_0}) H(\{J_j\}_{j\in K_0^c})\,, \\
        H(\{J_j\}_{j\in K_0^c}) & = &   \sum_{m=1}^{s-|K_0|}\sum_{K_1\sqcup\cdots\sqcup K_m}\left((-1)^{s-m}(m-1)!\prod_{i=1}^m(B_{X_i}-k)^{|K_i|}\right)\,.
    \end{eqnarray*}
    Observe that $H(\{J_j\}_{j\in K_0^c})$ is a polynomial in the variables $\{B_{J_j}\}_{j\in K_0^c\cup\{s\}}$ with integer coefficients. In particular, $H(\{J_j\}_{j\in K_0^c})|_{B_{J_s}=0}=0$ unless $K_0^c=\emptyset$. Indeed, for a fixed partition $K'_1,\ldots,K'_m$, there are~$m$ graphs in which the last component $X_m$ has $K_m=K'_i$ for some $i=1,\ldots,m$, and there is one graph with $m+2$ top components $X_0,\ldots,X_{m+1}$ such that $K_i=K'_i$ and $K_{m+1}=\emptyset$. The contributions of these graphs cancel. When $K_0^c=\emptyset$, there is only one graph with $m=1$ in which $X_1$ contains only $J_s$. For this graph,
    \begin{eqnarray*}
        t_{\Gamma}\text{Coeff}^{\Gamma}(J_1,\ldots,J_s)  & = & (-1)^{s-1}d_{i,\{s\}}G_k(J_0;J_1,\ldots,J_{s-1})\,,\\
        t_{\Gamma}\text{Coeff}^{\Gamma}(J_1,\ldots,J_s)|_{B_{J_s}=0}  & = & (-1)^{s-1}(a_i+k)G_k(J_0;J_1,\ldots,J_{s-1})|_{B_{J_s}=0}\,.
    \end{eqnarray*}
    
    Since $G_k(J_0;J_1,\ldots,J_{s-1})$ is a piecewise polynomial in $a_i,B_{J_0},\ldots, B_{J_s}$, where $a_j=(\sum_{j=0}^sB_{J_j})-a_i-2k$ as both $a_1,a_2$ appear in the definition of $G_k(J_0;J_1,\ldots,J_s)$, we conclude:
    \begin{eqnarray*}
       \text{Coeff}(J_1,\ldots,J_s)|_{B_{J_s}=0} & = & (-1)^s\sum_{i=1}^2 (a_i+k)G_k(J_0;J_1,\ldots,J_{s-1})|_{B_{J_s}=0} \\
       & = & (-1)^sG_k(J_0;J_1,\ldots,J_s)|_{B_{J_s}=0}\, .
    \end{eqnarray*}
    Hence, $B_{J_s}|\left[ \text{Coeff}(J_1,\ldots,J_s)-(-1)^sG_k(J_0;J_1,\ldots,J_s)\right]$. By symmetry, the same argument applies to any $B_{J_i}$. 
    Therefore, $\prod_{j=1}^s B_{J_j}\Big|\left[ \text{Coeff}(J_1,\ldots,J_s)-(-1)^sG_k(J_0;J_1,\ldots,J_s)\right]$. Since the left-hand side is a polynomial of degree at most $s-1$, it must be identically zero. This proves the desired claim.
    \end{proof}

\section{Flat geometric interpretation and examples}\label{sec:Flat}

In this section, we first sketch certain flat geometric ideas to prove Corollary~\ref{cor:polesk}, computing the isoresidual degree for strata of the form $\Omega^{k}\mathcal{M}_{0}(a_{1},a_{2},\rec[-k][p])$ with $k \geq 2$. Then we give an alternative formula in Section~\ref{sec:alterformula} and finally we give some exemples in Section~\ref{sec:ex}, both in the interesction theoretic setting and in the flat setting.
\smallskip
\par
\subsection{Some degenerations.}

First, given $\epsilon>0$, we define
$$U_{\epsilon} = \left\{(R_{1},\dots,R_{p})\in (\CC^{\ast})^{p} : |R_{j+1}| \leq \epsilon|R_{j}|\  \forall 1 \leq j \leq p-1 \right\}\,.$$
Recall that the resonance locus was introduced in Definition~\ref{defn:stratres}. The relevance of $U_{\epsilon}$ is established by the following result.

\begin{lem}\label{lem:UepsilonResonance}
Given a stratum $\Omega^{k}\mathcal{M}_{0}(a_{1},a_{2},\rec[-k][p])$ with $k \geq 2$, for sufficiently small $\epsilon$, the open set~$U_{\epsilon}$ is disjoint from the resonance locus.
\end{lem}

\begin{proof}
Consider $\xi \in \Omega^{k}\mathcal{M}_{0}(a_{1},a_{2},\rec[-k][p])$ whose $k$-residue tuple is $(R_{1},\dots,R_{p}) \in U_\epsilon$.
Let $r_{1},\dots,r_{p}$ be arbitrary $k$-th roots of $R_{1},\dots,R_{p}$, respectively, and let $w_{1},\dots,w_{p}$ be weights in 
$$W_{k}=\lbrace{ 0\rbrace} \cup \lbrace{ e^{\frac{2li\pi}{k}}~|~l \in \mathbb{Z}/k\mathbb{Z} \rbrace}\,,$$
not all equal to zero. As $\epsilon$ tends to $0$, the weighted sum $\sum\limits_{j=1}^{p} w_{j}r_{j}$ can be made arbitrarily close in modulus to $r_{j_{0}}$, where $w_{j_{0}}$ is the first nonzero weight in the sequence $w_{1},\dots,w_{p}$. Indeed, for $0<\epsilon<1$, we have $$\left| w_{j_{0}}r_{j_{0}}-\sum\limits_{j=1}^{p} w_{j}r_{j} \right | \leq |r_{j_{0}}|\sum\limits_{j > j_{0}}\epsilon^{j-j_{0}} \leq \frac{\epsilon}{1-\epsilon}|r_{j_{0}}|\,.$$
It follows that no such weighted sum can vanish when $\epsilon$ is sufficiently small. Hence, according to Definition~\ref{defn:stratres}, the residues $(R_{1},\dots,R_{p})$ lie outside the resonance locus.
\end{proof}

The local model of a pole of order $-k$ is the end of a cylinder bounded by one or several saddle connections  that meet at corners with an angle $\pi$. We refer to this cylinder as the \textit{domain} of the corresponding pole. We will show that, under generic conditions on the arguments of a configuration of $k$-residues lying in $U_{\epsilon}$, each such cylinder is bounded by a unique closed saddle connection incident to a unique conical singularity.

\begin{lem}\label{lem:uniqueconselle}
In a stratum $\Omega^{k}\mathcal{M}_{0}(a_{1},a_{2},\rec[-k][p])$ with $k \geq 2$, consider a $k$-differential realizing a configuration of $k$-residues $(R_{1},\dots,R_{p}) \in U_{\epsilon}$. For $\epsilon$ sufficiently small, the boundary of the domain of any pole with residue $R_{j}$, for $2 \leq j \leq p$, is formed by a closed saddle connection.
\end{lem}

\begin{proof}
According to Section~\ref{sub:Periods}, in a translation chart the period of a saddle connection between two conical singularities is of the form 
$$\pm \frac{1}{1-w_{0}}\sum\limits_{j=1}^{p} w_{j}r_{j}\,,$$ 
with nonzero $w_{0},w_{1},\dots,w_{p}$. For $\epsilon$ sufficiently small, the length of such a saddle connection is arbitrarily close to $|\frac{r_{1}}{1-w_{0}}|$, and is therefore larger than the circumference of any cylinder forming the domain of a pole with $k$-residue $R_{j}$ for $2 \leq j \leq p$. It follows that the boundary of these cylinders is formed by closed saddle connections and therefore contains only one conical singularity.
\end{proof}

We say that the cylinder of the $j$-th pole is \textit{attached} to the conical singularity $A_{i}$ (of order $a_{i}$) when the endpoints of its boundary saddle connection coincide with $A_{i}$.
\par
The $k$-differentials in the same isoresidual fiber over $U_{\epsilon}$ can thus be described in terms of the way these cylinders are attached to the conical singularities. We distinguish two cases, depending on  whether the corresponding flat surface has one or two conical singularities.

\subsection{The case $a_{2}<-k$: one conical singularity}\label{sub:oneCS}

For such strata, Corollary~\ref{cor:polesk} states that for a generic configuration of $k$-residues, the number of $k$-differentials in the corresponding isoresidual fiber is 
$$\frac{a_{1}!_{(k)}}{(a_{1}+k-kp)!_{(k)}}\,.$$ 
The following sketch of the proof parallels the one given in the abelian case in \cite[Section~4.1]{GeTaIso}.
\par
We consider a fiber over a configuration of $k$-residues in $U_{\epsilon}$ and proceed by induction on $p$. In the case $p=1$, the isoresidual fiber contains  only a single $k$-differential. The flat surface associated with this $k$-differential consists of a cone of angle $2\pi\frac{a_{1}+k}{k}$ to which a polar region with the given residue is attached.
\par
Now suppose the formula has been proved for $p-1$ poles. We then add a new pole of order~$-k$, whose residue is real and has magnitude at most $\epsilon$ times smaller than the other residues. We place the new pole at a separatrix of slope a multiple of $\frac{2\pi}{k}$ emanating from the singularity $z_1$ of order $a_{1}$ on the flat surfaces of the isoresidual fiber in the stratum $\komoduli[0](a_{1},a_{2};-b_{1},\dots,-b_{p-1})$ as follows. 

In terms of flat geometry, consider a separatrix $v$ of length $r_{p}$ starting at $z_1$. Note that the points lying on the line orthogonal to $v$ form a band that contains no conical singularities in its interior and intersects $z_{1}$ only at the starting point of $v$. We remove the upper half-infinite portion of this band and glue the two half-infinite boundary rays together by translation. Finally, we attach a half-infinite cylinder along $v$ to obtain the desired surface.
In the analytic picture, this corresponds to smoothing a twisted $k$-differential obtained by gluing a $k$-differential in $\komoduli[0](a_{1}+k;-k,-a_{1}-2k)$ to a differential in $\komoduli[0](a_{1},a_{2};-b_{1},\dots,-b_{p-1})$ in such a way that the residue at the pole of order~$-k$ is real.

On each flat surface to which we attach the horizontal pole, there are $a_{1}+k$ horizontal directions where we can place the new pole of order $-k$. Hence, the total number of possibilities is 
\begin{eqnarray*}
 (a_{1}+k) \frac{a_{1}!_{(k)}}{(a_{1}+k-k(p-1))!_{(k)}} &=& \frac{(a_{1}+k)!_{(k)}}{(a_{1}+k-k(p-1))!_{(k)}}\\
 &=& \frac{(a_{1}+k)!_{(k)}}{((a_{1}+k)+k-kp)!_{(k)}}\,.
\end{eqnarray*}

\subsection{The case $a_{1},a_{2}>-k$: two conical singularities}\label{sub:twoCS}

We now sketch two flat-geometric proofs.

\subsubsection{The first proof} We perform induction on the number $p$ of poles. One has that the degree is $1$ in any stratum $\Omega^{k}\mathcal{M}_{0}(a_{1},a_{2},-k)$ with $0>a_{1},a_{2}>-k$.
Suppose that Corollary~\ref{cor:polesk} is proved for $\Omega^{k}\mathcal{M}_{0}(a_{1},a_{2},\rec[-k][p])$. Consider an isoresidual fiber with the last residue $R_{p}$ real and such that the residues belong to $U_{\epsilon}$ for sufficiently small $\epsilon$. Let $R_{p}$ tend to zero. Note that this does not change the number of differentials in the fiber since we are in $U_{\epsilon}$. 

Now, either the pole was attached to $a_{1}$, in which case there are $\binom{p-2}{s-1} (a_{1}-k)!_{(k)} \cdot a_{2}!_{(k)}$ limits, or to $a_{2}$, in which case there are  $\binom{p-2}{s} a_{1}!_{(k)} \cdot (a_{2}-k)!_{(k)}$ limits. In the first case, we glue the pole of order $k$ at the zero of order $a_{1}-k$ in one of the $a_{1}$ horizontal directions. Similarly, in the second case, we glue it at the zero of order $a_{2}-k$.

Since there are $a_{i}$ horizontal directions at a zero of order $a_{i}-k$, the total number of such differentials is
$$\binom{p-2}{s-1} \cdot a_{1}\cdot (a_{1}-k)!_{(k)} \cdot a_{2}!_{(k)}  + \binom{p-2}{s} \cdot a_{2}\cdot a_{1}!_{(k)} \cdot (a_{2}-k)!_{(k)} = a_{1}!_{(k)} \cdot a_{2}!_{(k)} \cdot \binom{p-1}{s}\,.$$

\subsubsection{The second proof}

We introduce the notations $a_{1}=k \cdot d_{1}+\overline{a}_{1}$ and $a_{2}=k \cdot d_{2}+\overline{a}_{2}$, where $-k<\overline{a}_{1}<-\frac{k}{2}$ and $-\frac{k}{2}<\overline{a}_{2}<0$. We have $d_{1}+d_{2}=p-1$. The degree of the isoresidual cover is then 
$$
\binom{p-1}{d_{1}}
\cdot
\prod\limits_{1 \leq i \leq d_{1}} (\overline{a}_{1}+k\cdot i)
\cdot
\prod\limits_{1 \leq j \leq d_{2}} (\overline{a}_{2}+k\cdot j)\,.
$$

Consider a configuration $(R_{1},\dots,R_{p})$ of $k$-residues that belong to $U_{\epsilon}$ for some small $\epsilon$. We consider the path in $U_{\epsilon}$ defined by $(R_{1},tR_{2},\dots,tR_{p})$ with $t \in ]0,1]$. In the limit $t \rightarrow 0$, the $k$-differentials of the corresponding isoresidual fibers degenerate to multi-scale differentials whose dual graph is a cherry graph (each irreducible component is a complex line):
\begin{itemize}
\item A top component contains the pole of order $-k$ with $k$-residue $R_{1}$ and two nodal singularities of orders $\overline{a}_{1}$ and $\overline{a}_{2}$. 
\item A bottom component contains $z_{1}$, $d_{1}$ poles of order $-k$, and a nodal singularity of order $-2k-\overline{a}_{1}$;
\item Another bottom component contains $z_{2}$, $d_{2}$ poles of order $-k$, and a nodal singularity of order $-2k-\overline{a}_{2}$.
\end{itemize}

There is only one possible shape for the top component. We then apply the formula proved in Section~\ref{sub:oneCS} to count the possible configurations for the two bottom components. There are $\binom{p-1}{d_{1}}$ ways to split the $p-1$ poles of order $-k$ corresponding to the $k$-residues $tR_{2},\dots,tR_{p}$. Multiplying these factors yields the degree formula.

\subsection{An alternative formula}\label{sec:alterformula}

Given that a singularity of order $a>-k$ of a $k$-differential can be interpreted as a conical singularity of angle $\frac{a+k}{k}2\pi$ in the corresponding flat metric, the degree formula can be written as follows.

\begin{prop}
Assuming that $p \geq 1$, $a_{1},a_{2}>-k$, and that $a_{1}$ and $a_{2}$ are coprime with $k$, the number of $k$-differentials in $\Omega^{k}\mathcal{M}_{0}(a_{1},a_{2},\rec[-k][p])$ that realize a generic configuration of residues is 
$$
k^{p-1}
\cdot
\binom{p-1}{\lceil a_{1}/k \rceil +1}
\cdot
\frac{|\sin(\alpha_{1}\pi)|}{\pi}\Gamma(\alpha_{1})\Gamma(\alpha_{2})\,,
$$
where $2\alpha_{1}\pi$ and $2\alpha_{2}\pi$ are the conical angles corresponding to the singularities of orders $a_{1}$ and~$a_{2}$, respectively.
\end{prop}

\begin{proof}
We write $a_{i}=c_{i}k+f_{i}$ with $0<f_{i}<k$ (which is a convention different from that of Section~\ref{sub:twoCS}). The last part of Theorem~\ref{thm:main} shows that there are 
$$\binom{p-1}{\lceil a_{1}/k \rceil} \cdot a_{1}!_{(k)} \cdot a_{2}!_{(k)}$$ such differentials.
\par
Writing explicitly the $k$-factorial functions, we obtain $$
a_{i}!_{(k)} 
=
k^{c_{i}+1} \cdot \frac{x}{k}\cdot(\frac{a_{i}}{k}-1)\cdots(\frac{a_{i}}{k}-c_{i})\,.
$$
Combining this formula with the functional equation of the $\Gamma$ function, $z(z+1)\cdots(z+k)=\frac{\Gamma(z+k+1)}{\Gamma(z)}$, we obtain
$$
a_{i}!_{(k)}=k^{c_{i}+1}\frac{\Gamma\left(\frac{a_{i}}{k}+1\right)}{\Gamma\left(\frac{f_{i}}{k}\right)}\,.
$$
Keeping in mind that $c_{1}+c_{2}=p-3$ and $f_{1}+f_{2}=k$, we obtain 
$$
a_{1}!_{(k)} \cdot a_{2}!_{(k)}
=
k^{p+1}\frac{\Gamma\left(\frac{a_{1}}{k}+1\right) \cdot \Gamma\left(\frac{a_{2}}{k}+1\right)}{\Gamma\left(\frac{f_{1}}{k}\right)\cdot \Gamma\left(1-\frac{f_{1}}{k}\right)}\,.
$$
Using the formula $\Gamma(z) \cdot \Gamma(1-z)=\frac{\pi}{\sin(\pi z)}$ and replacing $\frac{a_{i}+k}{k}$ by $\alpha_{i}$, we obtain
$$
a_{1}!_{(k)} \cdot a_{2}!_{(k)}
=
k^{p-1}
\cdot
\frac{\sin\left(\frac{\pi f_{1}}{k}\right)}{\pi}\Gamma(\alpha_{1})\Gamma(\alpha_{2})\,.
$$
Finally, since $\alpha_{1}\pi \equiv \frac{f_{1}\pi}{k}~[\pi]$, we have $|\sin\left(\frac{\pi f_{1}}{k}\right)| = |\sin(\alpha_{1}\pi)|$. The claim follows.
\end{proof}

\subsection{Examples}\label{sec:ex}

We will describe the $k$-differentials in certain isoresidual fibers to illustrate the constructions and computations discussed in the preceding section.

In the quadratic case ($k=2$), we consider quadratic residues in $\R_{>0}$ to describe the differentials in the fiber using the incidence graphs introduced in \cite[Section 2.4]{getaquad}. These graphs are embedded in the sphere, with vertices corresponding to the components of the differential obtained by cutting along its saddle connections. For each saddle connection, there is an edge connecting the two corresponding vertices. In the case of two zeros, \cite[Lemma 6.5]{getaquad} shows that the graph is a cycle of odd length with trees attached at its vertices. Each face corresponds to a zero; when both zeros have the same order, we assume that $a_{1}$ corresponds to the compact face.

\begin{ex}
 Consider the strata $\quadomoduli[0](3,-1;-2,-2,-2)$ and  $\quadomoduli[0](1,1;-2,-2,-2)$. By Corollary~\eqref{cor:polesk}, the degree is~$3$ in the first case and $2$ in the second. We draw the incidence graphs of these differentials for $R_{1}>R_{2}>R_{3}$ in Figures~\ref{fig:-1,3} and~\ref{fig:1,1}, respectively. Note that the two quadratic differentials in Figure~\ref{fig:1,1} are isomorphic, but their vertex numbering differs. 
 
 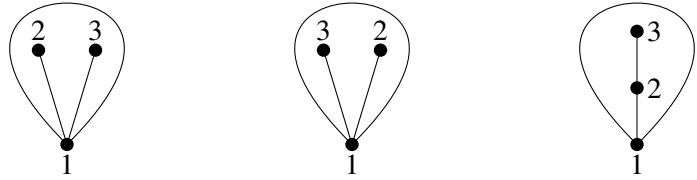
\begin{figure}[htb]
	\centering
\begin{tikzpicture} [scale=1.25]
 \begin{scope}
 \coordinate (a) at (0,0);\coordinate (b) at (-.3,1);\coordinate (c) at (.3,1);
 \draw (a) -- (b); \draw (a) -- (c);
 \draw (a) .. controls ++(45:1.5) and
++(0:.5) ..  (0,1.5) .. controls ++(180:.5) and
++(135:1.5) ..(a);
 \fill (a) circle (2pt); \fill (b) circle (2pt);\fill (c) circle (2pt);
 \node[below] at (a) {$1$};\node[above] at (b) {$2$};\node[above] at (c) {$3$};
\end{scope}
 \begin{scope}[xshift=3cm]
 \coordinate (a) at (0,0);\coordinate (b) at (-.3,1);\coordinate (c) at (.3,1);
 \draw (a) -- (b); \draw (a) -- (c);
 \draw (a) .. controls ++(45:1.5) and
++(0:.5) ..  (0,1.5) .. controls ++(180:.5) and
++(135:1.5) ..(a);
 \fill (a) circle (2pt); \fill (b) circle (2pt);\fill (c) circle (2pt);
 \node[below] at (a) {$1$};\node[above] at (b) {$3$};\node[above] at (c) {$2$};
\end{scope}
 \begin{scope}[xshift=6cm]
 \coordinate (a) at (0,0);\coordinate (b) at (0,.6);\coordinate (c) at (0,1.2);
 \draw (a) -- (b); \draw (b) -- (c);
 \draw (a) .. controls ++(45:1.5) and
++(0:.5) ..  (0,1.5) .. controls ++(180:.5) and
++(135:1.5) ..(a);
 \fill (a) circle (2pt); \fill (b) circle (2pt);\fill (c) circle (2pt);
 \node[below] at (a) {$1$};\node[right] at (b) {$2$};\node[right] at (c) {$3$};
\end{scope}
\end{tikzpicture}
  \caption{Quadratic differentials in the stratum $\quadomoduli[0](3,-1;-2,-2,-2)$ with generic residues.}
 \label{fig:-1,3}
\end{figure}
 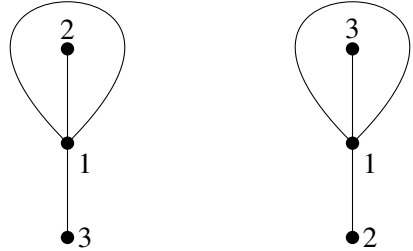
\begin{figure}[htb]
	\centering
\begin{tikzpicture} [scale=1.25]
 \begin{scope}
 \coordinate (a) at (0,0);\coordinate (b) at (0,1);\coordinate (c) at (0,-1);
 \draw (a) -- (b); \draw (a) -- (c);
 \draw (a) .. controls ++(45:1.5) and
++(0:.5) ..  (0,1.5) .. controls ++(180:.5) and
++(135:1.5) ..(a);
 \fill (a) circle (2pt); \fill (b) circle (2pt);\fill (c) circle (2pt);
 \node[below right] at (a) {$1$};\node[above] at (b) {$2$};\node[right] at (c) {$3$};
\end{scope}
 \begin{scope}[xshift=3cm]
 \coordinate (a) at (0,0);\coordinate (b) at (0,1);\coordinate (c) at (0,-1);
 \draw (a) -- (b); \draw (a) -- (c);
 \draw (a) .. controls ++(45:1.5) and
++(0:.5) ..  (0,1.5) .. controls ++(180:.5) and
++(135:1.5) ..(a);
 \fill (a) circle (2pt); \fill (b) circle (2pt);\fill (c) circle (2pt);
 \node[below right] at (a) {$1$};\node[above] at (b) {$3$};\node[right] at (c) {$2$};
\end{scope}

\end{tikzpicture}
  \caption{Quadratic differentials in the stratum $\quadomoduli[0](1,1;-2,-2,-2)$ with generic residues.}
 \label{fig:1,1}
\end{figure}
\end{ex}

\begin{ex}
 We consider the stratum $\Omega^{2}\mathcal{M}_{0}(1,3,[-2]^{4})$. In this case, we have $a_{1}!! \cdot a_{2}!!=3$ and $\binom{p-1}{\lceil a_{1}/k \rceil}= \binom{3}{1}=3$. Therefore, by Corollary~\eqref{cor:polesk}, the degree of the isoresidual cover is~$9$.

 Now, using Equation~\eqref{eq:degre1} with $a_{1}=1$, we consider the subsets $I=\emptyset$, for which $c_{1,I}=3$, and $I=\lbrace1\rbrace,\lbrace2\rbrace,\lbrace3\rbrace,\lbrace4\rbrace$, for which $c_{1,I}=1$.
Hence, the degree is
  $$3 \cdot f_2(1,1) \cdot f_2(3,5) + 4 \cdot 1 \cdot f_2(1,2) \cdot f_2(3,4) = 3  \cdot\frac{1}{3}\cdot (-3) + 4\cdot 1 \cdot 1 \cdot 3 = 9\,.$$
  Following the construction, the corresponding differentials are obtained by gluing one vertex to the graphs from Figures~\ref{fig:-1,3} and~\ref{fig:1,1}. These are illustrated in Figure~\ref{fig:1,3} in the case where $R_{1}>R_{2}>R_{3}>R_{4}$ are real, and each residue is much larger than the subsequent one.
   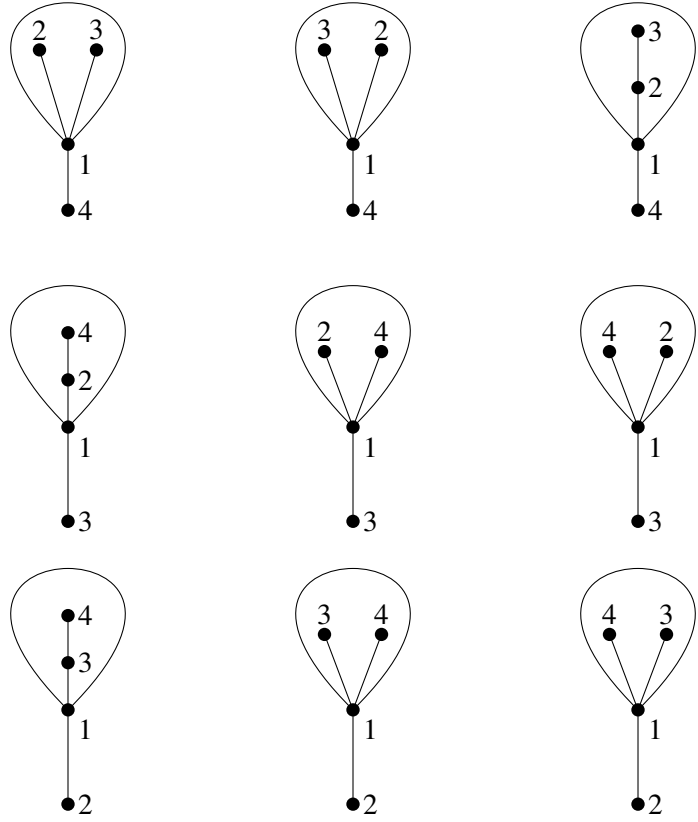
\begin{figure}[htb]
	\centering
\begin{tikzpicture} [scale=1.25]
 \begin{scope}
 \coordinate (a) at (0,0);\coordinate (b) at (-.3,1);\coordinate (c) at (.3,1); \coordinate (d) at (0,-.7);
 \draw (a) -- (b); \draw (a) -- (c);\draw (a) -- (d);
 \draw (a) .. controls ++(45:1.5) and
++(0:.5) ..  (0,1.5) .. controls ++(180:.5) and
++(135:1.5) ..(a);
 \fill (a) circle (2pt); \fill (b) circle (2pt);\fill (c) circle (2pt);\fill (d) circle (2pt);
 \node[below right] at (a) {$1$};\node[above] at (b) {$2$};\node[above] at (c) {$3$};\node[right] at (d) {$4$};
\end{scope}
 \begin{scope}[xshift=3cm]
 \coordinate (a) at (0,0);\coordinate (b) at (-.3,1);\coordinate (c) at (.3,1); \coordinate (d) at (0,-.7);
 \draw (a) -- (b); \draw (a) -- (c);\draw (a) -- (d);
 \draw (a) .. controls ++(45:1.5) and
++(0:.5) ..  (0,1.5) .. controls ++(180:.5) and
++(135:1.5) ..(a);
 \fill (a) circle (2pt); \fill (b) circle (2pt);\fill (c) circle (2pt);\fill (d) circle (2pt);
 \node[below right] at (a) {$1$};\node[above] at (b) {$3$};\node[above] at (c) {$2$};\node[right] at (d) {$4$};\end{scope}
 \begin{scope}[xshift=6cm]
 \coordinate (a) at (0,0);\coordinate (b) at (0,.6);\coordinate (c) at (0,1.2);\coordinate (d) at (0,-.7);
 \draw (a) -- (b); \draw (b) -- (c);\draw (a) -- (d);
 \draw (a) .. controls ++(45:1.5) and
++(0:.5) ..  (0,1.5) .. controls ++(180:.5) and
++(135:1.5) ..(a);
 \fill (a) circle (2pt); \fill (b) circle (2pt);\fill (c) circle (2pt);\fill (d) circle (2pt);
 \node[below right] at (a) {$1$};\node[right] at (b) {$2$};\node[right] at (c) {$3$};\node[right] at (d) {$4$};
\end{scope}

 \begin{scope}[yshift=-3cm]
 \coordinate (a) at (0,0);\coordinate (b) at (0,.5);\coordinate (c) at (0,-1);\coordinate (d) at (0,1);
 \draw (a) -- (b); \draw (a) -- (c);\draw (b) -- (d);
 \draw (a) .. controls ++(45:1.5) and
++(0:.5) ..  (0,1.5) .. controls ++(180:.5) and
++(135:1.5) ..(a);
 \fill (a) circle (2pt); \fill (b) circle (2pt);\fill (c) circle (2pt);\fill (d) circle (2pt);
 \node[below right] at (a) {$1$};\node[right] at (b) {$2$};\node[right] at (c) {$3$};\node[right] at (d) {$4$};
\end{scope}

 \begin{scope}[yshift=-3cm,xshift=3cm]
 \coordinate (a) at (0,0);\coordinate (b) at (-.3,.8);\coordinate (c) at (0,-1);\coordinate (d) at (.3,.8);
 \draw (a) -- (b); \draw (a) -- (c);\draw (a) -- (d);
 \draw (a) .. controls ++(45:1.5) and
++(0:.5) ..  (0,1.5) .. controls ++(180:.5) and
++(135:1.5) ..(a);
 \fill (a) circle (2pt); \fill (b) circle (2pt);\fill (c) circle (2pt);\fill (d) circle (2pt);
 \node[below right] at (a) {$1$};\node[above] at (b) {$2$};\node[right] at (c) {$3$};\node[above] at (d) {$4$};
\end{scope}

\begin{scope}[yshift=-3cm,xshift=6cm]
\coordinate (a) at (0,0);\coordinate (b) at (-.3,.8);\coordinate (c) at (0,-1);\coordinate (d) at (.3,.8);
 \draw (a) -- (b); \draw (a) -- (c);\draw (a) -- (d);
 \draw (a) .. controls ++(45:1.5) and
++(0:.5) ..  (0,1.5) .. controls ++(180:.5) and
++(135:1.5) ..(a);
 \fill (a) circle (2pt); \fill (b) circle (2pt);\fill (c) circle (2pt);\fill (d) circle (2pt);
 \node[below right] at (a) {$1$};\node[above] at (b) {$4$};\node[right] at (c) {$3$};\node[above] at (d) {$2$};
\end{scope}

 \begin{scope}[yshift=-6cm]
 \coordinate (a) at (0,0);\coordinate (b) at (0,.5);\coordinate (c) at (0,-1);\coordinate (d) at (0,1);
 \draw (a) -- (b); \draw (a) -- (c);\draw (b) -- (d);
 \draw (a) .. controls ++(45:1.5) and
++(0:.5) ..  (0,1.5) .. controls ++(180:.5) and
++(135:1.5) ..(a);
 \fill (a) circle (2pt); \fill (b) circle (2pt);\fill (c) circle (2pt);\fill (d) circle (2pt);
 \node[below right] at (a) {$1$};\node[right] at (b) {$3$};\node[right] at (c) {$2$};\node[right] at (d) {$4$};
\end{scope}

 \begin{scope}[yshift=-6cm,xshift=3cm]
 \coordinate (a) at (0,0);\coordinate (b) at (-.3,.8);\coordinate (c) at (0,-1);\coordinate (d) at (.3,.8);
 \draw (a) -- (b); \draw (a) -- (c);\draw (a) -- (d);
 \draw (a) .. controls ++(45:1.5) and
++(0:.5) ..  (0,1.5) .. controls ++(180:.5) and
++(135:1.5) ..(a);
 \fill (a) circle (2pt); \fill (b) circle (2pt);\fill (c) circle (2pt);\fill (d) circle (2pt);
 \node[below right] at (a) {$1$};\node[above] at (b) {$3$};\node[right] at (c) {$2$};\node[above] at (d) {$4$};
\end{scope}

\begin{scope}[yshift=-6cm,xshift=6cm]
\coordinate (a) at (0,0);\coordinate (b) at (-.3,.8);\coordinate (c) at (0,-1);\coordinate (d) at (.3,.8);
 \draw (a) -- (b); \draw (a) -- (c);\draw (a) -- (d);
 \draw (a) .. controls ++(45:1.5) and
++(0:.5) ..  (0,1.5) .. controls ++(180:.5) and
++(135:1.5) ..(a);
 \fill (a) circle (2pt); \fill (b) circle (2pt);\fill (c) circle (2pt);\fill (d) circle (2pt);
 \node[below right] at (a) {$1$};\node[above] at (b) {$4$};\node[right] at (c) {$2$};\node[above] at (d) {$3$};
\end{scope}

\end{tikzpicture}
  \caption{Quadratic differentials in the stratum $\quadomoduli[0](1,3;-2,-2,-2,-2)$ with generic residues.}
 \label{fig:1,3}
\end{figure}
\par
Now we can deform these differentials until they reach a  resonance hyperplane. For example, let us increase $R_{4}$ until we obtain $P(R_{3},R_{4})=0$, meaning that $r_{4}+r_{3}=0$. In this case, we see that the only graph becoming singular is the one in the lower-left corner. Therefore, the corresponding isoresidual fiber has $8$ elements in this situation. 
\end{ex}

In the following examples, we illustrate the formulas from Theorem~\ref{thm:one-resonant} and~\ref{thm:iso-arbitrary} on specific configurations of residues for cubic and quartic differentials. For each of these configurations, the isoresidual fiber is known to be empty according to the classification of obstructions \cite[Theorem 1.6]{GT-k-residue}. These examples thus provide a consistency check for our formula in the presence of known obstructions to the residue realization problem (and vice versa).
\begin{ex} Consider the case
$k=3$, $\mu=(4,-1;-3,-3,-3)$, and $\mathcal{R}=[1:1:1]$.
Then $c_{1,I}>0$ holds for the following subsets:
\begin{itemize}
\item $I=\emptyset$ with $c_{1,I}=7$,
\item $I=\{i\}$ for $i=1,2,3$ with $c_{1,I}=4$,
\item $I=\{i_1,i_2\}$ with $1\leq i_1<i_2\leq 3$ with $c_{1,I}=1$.
\end{itemize}
We have
\begin{eqnarray*}
    d_3(\mu) & = &[7f_3(4,1)f_3(-1,4)]+ 3[4f_3(4,2)f_3(-1,3)]+3[1f_3(4,3)f_3(-1,2)]\\
    & = & [7\cdot (1/7) \cdot (-1)(-4)]+3[4\cdot1\cdot(-1)]+3[1\cdot 4 \cdot 1]\\
    & = & 4\, .
\end{eqnarray*}
There is a single resonance subset $I=\{1,2,3\}$, as mentioned in the proof above, for which $f_I=f(9/3-1,4)=2$ and $\Ab_{\mathcal{R}}(I)=2$. Hence, the cardinality of the isoresidual fiber is
\begin{equation*}
    d_3(\mu)-f_I\cdot \Ab_{\mathcal{R}}(I) = 0\,.
\end{equation*}
The fact that this isoresidual fiber is empty is consistent with \cite[Theorem 1.6 (2)]{GT-k-residue}.
\end{ex}

\begin{ex} \label{ex:451444}
Let $k=4$, consider the partition $\mu=(5,-1,-4,-4,-4)$ and $\mathcal{R}=[1:1:-4]$. The generic isoresidual fiber has $d_4(\mu)=5$. There are two resonant subsets: $J_1=\{1,2\}$ and $J'_1=\{1,2,3\}$.
    
For $J_1=\{1,2\}$, we have $J_0=[3]$. The only nonzero term for $G_4(J_0;J_1)$ occurs when $I={1}$ and $I^c=\emptyset$, with $d_{1,I}=1>0$ and $d_{2,I^c}=3>0$. Then $\mu_{I,I^c}=[-3,0,-4]$ and $d_4(\mu_{I,I^c})=1$. We compute: 
\begin{eqnarray*}
    G_4(J_0;J_1) & = & 1\cdot 3 \cdot 1 \cdot (5+4)^0(-1+4)^{-1} = 1\,,\\
    f_{J_1} & = & f(8/4-1,3) = 1\,,\\
    \Ab_{\mathcal{R}}(J_1) & = & \#\{[r_1:r_2]=[1:-1]\} = 1,
\end{eqnarray*}
so the contribution is
\begin{equation*}
    (-1)G_4(J_0;J_1)f_{J_1}Ab_{\mathcal{R}}(J_1) = -1\,.
\end{equation*}

For $J'_1=\{1,2,3\}$, we have $J'_0=\emptyset$. The only nonzero term for $G_4(J'_0;J'_1)$ occurs when $I=\emptyset$ with $d_{1,I}=9>0$. We compute: 
\begin{eqnarray*}
    G_4(J'_0;J'_1) & = & 9 \cdot (5+4)^{-1}(-1+4)^{0} = 1\,,\\
    f_{J'_1} & = & f(12/4-1,4) = 2\,,\\
    \Ab_{\mathcal{R}}(J'_1) & = & \#\{[r_1:r_2:r_3]=[1:i:-1-i],[1,-i,-1+i]\} = 2
\end{eqnarray*}
so the contribution is
\begin{equation*}
    (-1)G_4(J'_0;J'_1)f_{J'_1}Ab_{\mathcal{R}}(J'_1)= -4\,.
\end{equation*}
We conclude that the special isoresidual fiber has cardinality $5-1-4=0$, so it is empty, consistent with the known obstruction.  
\end{ex}

\begin{ex} \label{ex:133444444}
Let $k=4$ and consider the partition $\mu=(13,3,\rec[-4][6])$, and $\mathcal{R}=[\rec[1][6]]$. The generic isoresidual fiber has  $d_4(\mu)=8775$. The resonant subsets are generated by all 2-element subsets $\{[i_1,i_2]\}_{1\leq i_1<i_2\leq 6}$ of $\{1,\ldots, 6\}$. The contributions are computed as follows:
\begin{itemize}
    \item Single $2$-element subsets $J_1=[i_1,i_2]$ (15 subsets): 
    \begin{eqnarray*}
    (-1)G_4(J_0;J_1)f_{J_1} & = & -405\,, \\ 
    \Ab_{\mathcal{R}}(J_1) & = & \#\{[r_{i_1}:r_{i_2}]=[1:-1]\}=1\,.
    \end{eqnarray*}
    The total contribution of this type is $-405\cdot 15=-6075$.
    \item Single $4$-element subsets $J_1=[i_1,i_2,i_3,i_4]$ (15 subsets): 
    \begin{eqnarray*}
    (-1)G_4(J_0;J_1)f_{J_1} & = & -18\,, \\
    \Ab_{\mathcal{R}}(J_1)& = &\#\left\{[r_{i_1}:r_{i_2}:r_{i_3}:r_{i_4}]=[1:1:-1:-1]^{\times 3},[1:-1:i:-i]^{\times 6}\right\} \\
    &= &9\,.
    \end{eqnarray*}
    The notation $[]^{\times n}$ indicates that there are $n$ tuples of abelian residues of this type, obtained by permuting all entries except the first, after quotienting by $\mathbb{C}^*$. The total contribution of this type is $-2430$.
    \item Full set $J_1=[1,2,3,4,5,6]$: 
    \begin{eqnarray*}
    (-1)G_4(J_0;J_1)f_{J_1} & = & -120\,, \\
    \Ab_{\mathcal{R}}(J_1) & = & \#\left\{[r_1:\cdots : r_6]=[1^3:-1^3]^{\times 10},[1^2:-1^2:i:-i]^{\times 60},[1:-1:i^2:-i^2]^{\times 30}\right\} \\
    &= & 100\,.
    \end{eqnarray*}
    The total contribution of this type is $-12000$.
    \item $J_1,J_2=[i_1,i_2],[i_3,i_4]$ ($45$ pairs):  
   \begin{eqnarray*}
    (-1)^2G_4(J_0;J_1,J_2)f_{J_1}f_{J_2} & = & 51\,,\\
    \Ab_{\mathcal{R}}(J_j)& = & 1
    \end{eqnarray*}
   for $j=1,2$, as we previously computed. The total contribution of this type is $2295$.
    \item $J_1,J_2=[i_1,i_2,i_3,i_4],[i_5,i_6]$ ($15$ pairs): 
    \begin{eqnarray*}
     (-1)^2G_4(J_0;J_1,J_2)f_{J_1}f_{J_2} & = & 102\,, \\
     \Ab_{\mathcal{R}}(J_1) & = & 9\,,\\
     \Ab_{\mathcal{R}}(J_2) & = & 1\,.
   \end{eqnarray*}
     The total contribution of this type is $13770$. 
    \item $J_1,J_2,J_3=[i_1,i_2],[i_3,i_4],[i_5,i_6]$ ($15$ combinations): 
    \begin{eqnarray*}
    (-1)^3G_4(J_0;J_1,J_2,J_3)\prod_{j=1}^3f_{J_j} & = & -289\,, \\ 
    \Ab_{\mathcal{R}}(J_j) & = & 1 
    \end{eqnarray*}
    for $j=1,2,3$. The total contribution of this type is $-4335$.
\end{itemize}
Adding everything together, we conclude that the special isoresidual fiber has cardinality $0$, and is therefore empty. 
\end{ex}

\section{Applications to spherical geometry}\label{sec:Spherical}

The generalization of the uniformization theorem to  surfaces endowed with metrics having conical singularities of prescribed angles remains an open problem in the case of metrics with constant positive curvature (see \cite{Er1} for a general survey). This problem has been approached using various methods, including PDEs (see, for example, \cite{CL}), hypergeometric functions (see \cite{EGT}), and algebraic geometry (see \cite{LSX}). Recently, progress has been made by considering specific subclasses of cone spherical metrics with restricted monodromy. In particular, when the monodromy is \textit{coaxial}---meaning that the monodromy group, a subgroup of $\text{SO}(3)$, is confined to a one-parameter family of rotations around a single axis---Eremenko \cite{Erm} obtained a complete characterization of the configurations of conical angles that can be realized by such metrics.
\par
A slightly larger class of cone spherical metrics has also been considered. A metric is said to have  \textit{dihedral monodromy} if its monodromy preserves a great circle of the model sphere. It was shown in \cite[Theorem~1.2]{SongSpher} that these metrics are defined by (multi-valued) developing maps of the form $z \mapsto \exp\left(\int^{z} \sqrt{q}\right)$, where $q$ is a meromorphic quadratic differential with at most double poles and such that every period of $\sqrt{q}$ is real. In particular, the quadratic residue at each double pole must be a positive real number. The corresponding spherical metric is obtained as the pullback of $\frac{4|dw|^{2}}{(1+|w|^{2})^{2}}$. The interpretation of the singularities of $q$ is as follows:
\begin{itemize}
    \item a zero of order $a$ corresponds to a conical singularity of angle $(2+a)\pi$;
    \item a simple pole corresponds to a conical singularity of angle $\pi$;
    \item a double pole with quadratic residue $\frac{\theta^{2}}{4\pi^{2}}$ corresponds to a conical singularity of angle $\theta$.
\end{itemize}
A complete characterization of the configurations of conical angles realized by metrics with dihedral monodromy was obtained in \cite{getasphere}. More recently, the topology of the moduli space of such metrics with constrained monodromy has been studied in \cite{LX}.
\par
For metrics on the sphere with dihedral monodromy, certain configurations of conical angles can be realized by only finitely many distinct cone spherical metrics. In this case, the counting problem reduces to enumerating  the number of elements in an isoresidual fiber. Specifically, for quadratic differentials with at most double poles, Theorem~\ref{thm:main} provides such a count.

\begin{thm}\label{thm:Spherical}
Consider $n \geq 1$,  two odd integers $a,b$ satisfying $a+b=2n$, and non-integer positive real numbers $c_{1},\dots,c_{n}$ such that any sum of the form $\sum\limits_{i=1}^{n} \epsilon_{i} c_{i}$ with $\epsilon_{1},\dots,\epsilon_{n} \in \lbrace{ -1,0,1 \rbrace}$ vanishes only if $\epsilon_{1}=\dots=\epsilon_{n}=0$.

Then the number of distinct cone spherical metrics with dihedral monodromy on the sphere, with $n+2$ marked conical singularities of angles $a\pi,b\pi,c_{1}\pi,\dots,c_{n}\pi$, is 
$$
\binom{n-1}{\frac{a -1}{2}}(a-2)!!\cdot (b-2)!!\,.$$
\end{thm}
Note that, since $a+b=2n$, the formula in the theorem remains unchanged if $a$ is replaced by~$b$. Moreover, the hypothesis that $a$ and $b$ are odd implies that all these metrics are strictly dihedral; that is, they are not coaxial.

\begin{proof}
Given such a cone spherical metric on the sphere, there exists a quadratic differential~$q$ whose developing map is of the form $z \mapsto \exp\left(\int^{z} \sqrt{q}\right)$. The conical singularities with non-integer angles $c_{1}\pi,\dots,c_{n}\pi$ cannot correspond to simple poles or zeros of $q$; therefore, they must be double poles of~$q$. 

A conical singularity of angle $a\pi$ or $b\pi$ cannot correspond to a double pole of $q$, because the other would then have to be a zero of even order (the total order of the singularities of $q$ is $-4$). It follows that these two conical singularities of angles $a\pi$ and $b\pi$ correspond to zeros of $q$ of orders $a-2$ and $b-2$, respectively. 

Any other zero or simple pole of $q$ would correspond to  another conical singularity with angle in $\pi\mathbb{N}$. Since $a+b=2n$, the total order of the singularities corresponding to the conical singularities of the spherical metric is $-4$, so there are no additional singularities. 

Thus, the quadratic differential $q$ belongs to the stratum $\Omega^{2}\mathcal{M}_{0}(a-2,b-2,\rec[-2][n])$. The remaining task is to count the number of quadratic differentials in this stratum with $n$ double poles whose  quadratic residues are $\left(\frac{c_{1}}{2}\right)^{2}, \dots, \left(\frac{c_{n}}{2}\right)^{2}$.
\par
The genericity hypothesis on $c_{1},\dots,c_{n}$ ensures  that the configuration $\left(\frac{c_{1}}{2}\right)^{2}, \dots, \left(\frac{c_{n}}{2}\right)^{2}$ of quadratic residues lies outside the resonance locus. By Theorem~\ref{thm:main}, this allows us to count the number of elements in the corresponding isoresidual fiber. By hypothesis, all residues are pairwise distinct. However, it is possible that $a=b$; in this case, the corresponding singularities of the quadratic differentials are marked, which is consistent with our choice to work with marked strata throughout this paper.
\end{proof}

\printbibliography
\end{document}